\documentclass{article}
\usepackage{graphicx}
\graphicspath{graphics/}
\usepackage{amsfonts}
\usepackage{mathrsfs}
\usepackage{float}
\usepackage{bm}
\usepackage{nicematrix}
\usepackage{mathdots}
\usepackage{array}
\usepackage{tikz}
\usepackage{verbatim}
\usepackage{tikz-cd}
\usepackage{booktabs}
\usepackage{enumitem}
\usepackage{amsmath}
\usepackage{amssymb}
\usepackage{amsthm}

\theoremstyle{plain}
\newtheorem{theorem}{Theorem}[section]
\newtheorem{proposition}[theorem]{Proposition}
\newtheorem{lemma}[theorem]{Lemma}
\newtheorem{corollary}[theorem]{Corollary}

\theoremstyle{definition}

\theoremstyle{remark}
\newtheorem{remark}[theorem]{Remark}

\newtheorem{example}[theorem]{Example}
\allowdisplaybreaks

\title{Complexity of quantum cohomology}
\author{Xiaobo Liu \thanks{Research was partially supported by NSFC grants 12341105 and 12526302.}, \,\,\,  Chongyu Wang}
\date{}

\begin{document}

\maketitle

\begin{abstract}
      Circuit complexity for two-dimensional topological quantum field theories (2D TQFT) was defined by Couch, Fan, and Shashi in \cite{CFS}. In this paper, we study complexity for the 2D TQFT given by quantum cohomology of compact symplectic manifolds. We will estimate the number of states with finite approximate complexity of arbitrarily small tolerance for Fano complete intersections and (co)minuscule homogeneous varieties. We will give an upper bound for the dimension of the space spanned by states with finite complexity. In the case of ${\rm Gr}(2,n)$, this bound is sharp and we also obtain a precise description for this subspace. For (co)minuscule homogeneous varieties, we prove a positivity result for the eigenvalues of quantum multiplication by the handle element (also called the quantum Euler class) divided by the class of a point.
\end{abstract}

\section{Introduction}

In quantum computation, the circuit complexity of a unitary operator (quantum algorithm) is defined to be the minimal number of elementary operators (called gates) needed to construct the given operator. In \cite{HH}, Harlow and Hayden proposed to use quantum complexity to study black holes. Complexity also plays an important role in the study of AdS/CFT duality (see, for example, \cite{S}, \cite{SS}, \cite{BRSSZ}). This raised the question of how to define complexity for conformal field theories. A special class of conformal field theories are given by two-dimensional topological quantum field theories (2D TQFT) (see, for example, \cite{RFFS}).
In \cite{CFS}, Couch, Fan and Shashi defined a notion of circuit complexity for 2D TQFT.
The purpose of this paper is to study complexity of 2D TQFT given by quantum cohomology of compact symplectic manifolds.

The quantum cohomology of a compact symplectic manifold $X$ is a deformation of ordinary cohomology ring of $X$, defined by
genus-$0$ Gromov-Witten invariants of $X$ (cf. \cite{RT} and \cite{LiT}). In this paper, we will use $H^*(X)$ to denote
the cohomology ring of $X$ with coefficient $\mathbb{C}$. Quantum cohomology gives a Frobenius manifold structure on $H^*(X)$ (cf. \cite{Du} and \cite{Ma}).  In particular, it gives a family of Frobenius algebra structures on $H^*(X)$
(Rigorously speaking, this is true for Fano varieties. In general, we should work over the Novikov ring of $X$ instead of
$\mathbb{C}$). It is well known that 2D TQFTs are one-to-one correspondent to Frobenius algebras (see, for example, \cite{K}). Therefore quantum cohomology of $X$ gives a family of 2D TQFTs and we have a natural notion of complexity for quantum cohomology following the definition given in \cite{CFS}. In this paper, we mainly consider complexity for small quantum cohomology which is defined by $3$-point genus-0 Gromov-Witten invariants.

For a general 2D TQFT, the complexity is defined using the handle operator of the corresponding Frobenius algebra,
which serves as the unique elementary gate in the definition of circuit complexity. In the case of quantum cohomology, the {\it handle operator} is given by the quantum multiplication of an element $\Delta \in H^*(X)$
which has the form
\begin{equation} \label{eqn:Delta}
\Delta :=\sum_{i,j} g^{ij} e_i * e_j,
\end{equation}
where $\{ e_i \}$ is a basis of $H^*(X)$, $g^{ij}$ are entries for the inverse matrix of
Poincar\'e intersection pairing on $H^*(X)$ with respect to this basis, and $*$ in equation \eqref{eqn:Delta} is the quantum multiplication.
A 1-party {\it state} for the 2D TQFT given by the quantum cohomology is an element in $\mathbb{P}H^*(X)$, which is the projectivization of $H^*(X)$. For any $z \in H^*(X)$, we use $[z] \in \mathbb{P}H^*(X)$
to represent the image of $z$ under the natural projection from $H^*(X)$ to $\mathbb{P}H^*(X)$.
Define $\Delta * [z] := [\Delta * z]$ for $z \in H^*(X)$ and $\Delta * z \neq 0$. Fix a reference state
$S_0 \in \mathbb{P}H^*(X)$. The (exact) {\it complexity} of any state $S \in \mathbb{P}H^*(X)$, denoted by
${\cal C}_{S_0}(S)$,  is defined to be
the smallest non-negative integer $k$ such that $\Delta^{*k} * S_0 = S$, where
$\Delta^{*k}$ is the $k$-th quantum power of $\Delta$ (i.e. quantum multiplication of $\Delta$ with itself $k$ times). If such $k$ does not exist, we define
${\cal C}_{S_0}(S) = \infty$.

Since the set of states with finite complexity is always a countable set, it follows that most states have infinite complexity. So it is necessary to consider {\it approximate complexity} ${\cal C}_{S_0}(S; \epsilon)$ with {\it tolerance} $\epsilon > 0$, which is defined to be the smallest non-negative integer $k$ such that
$d(\Delta^{*k} * S_0,  S) \leq \epsilon$, where $d$ is the distance between two states defined by a fixed metric on $\mathbb{P}H^*(X)$. It is an interesting question how large is the set of states which have finite
approximate complexity with arbitrarily small tolerance. More precisely, let
\begin{equation} \label{eqn:SInfty}
\mathfrak{S}_\infty := \{ S \in \mathbb{P}H^*(X) \mid {\cal C}_{S_0}(S) = \infty \,\,\, {\rm and}
        \,\,\, {\cal C}_{S_0}(S; \epsilon) < \infty \,\,\, {\rm for \,\,\, all \,\,\,} \epsilon > 0 \}.
\end{equation}
This set is contained in the closure of the set of all states with finite complexity and therefore is
independent of the choice of the metric on $\mathbb{P}H^*(X)$.
We would like to investigate how large is $\mathfrak{S}_\infty$.
Note that the set of states with finite exact complexity is usually infinite (see, for example, the
case of quadrics in Section \ref{subsect: even quadrics}). It could happen that its limiting set $\mathfrak{S}_\infty$ might be even larger. In fact, as shown in \cite{CFS}, for a class of semisimple 2D TQFT, $\mathfrak{S}_\infty$ may contain a torus with positive dimension and is therefore uncountable.
In this paper, we will show  that $\mathfrak{S}_\infty$ turns out to be very small for quantum cohomology of Fano complete intersections and (co)minuscule homogeneous varieties. These homogeneous varieties include Grassmannians, quadrics, Lagrangian Grassmannians, orthogonal Grassmannians, Cayley plane, and Freudenthal variety (cf. \cite{BL}).
\begin{theorem}\label{thm: main thm 1}
    For the small quantum cohomology of all (co)minuscule homogeneous varieties and Fano complete intersections of complex dimension bigger than $2$, $\mathfrak{S}_\infty$ is
    a finite set for arbitrary reference state. Moreover, the number of points in $\mathfrak{S}_\infty$ is less than or equal to the order of the point class
    for (co)minuscule homogeneous varieties and is less than or equal to 2 for Fano complete intersections.
\end{theorem}
\noindent
This theorem covers both semisimple and non-semisimple examples. In fact,
it is proved in \cite{CMP3} that the small quantum cohomology of minuscule and cominuscule homogeneous varieties is always semisimple.
But there are many Fano complete intersections whose small quantum cohomology is not semisimple (see Remark \ref{rm: semisimp of fci}).
Theorem \ref{thm: main thm 1} is a combination of Theorem \ref{thm: main thm for finiteness} and
Theorem \ref{thm:ComplexityFCI}. It is an interesting question whether this finiteness result holds for a larger class of manifolds, for example, for all Fano varieties or all manifolds with semisimple quantum
cohomology.

As suggested in \cite{CFS}, a natural choice for the reference state is $S_0= [\mathbf{1}]$, where $\mathbf{1}$ is the identity of the ordinary cohomology ring and is also the unit element of the Frobenius algebra defined by quantum cohomology. In this case, the set of all states with finite complexity is precisely $\{ [\Delta^{*k}] \mid k \geq 0 \}$.
Besides complexity, there are also many other interests in studying $\Delta^{*k}$. In small quantum cohomology, the handle element $\Delta$ was first studied by  Abrams in \cite{A}, and it was called the quantum Euler class (In this paper we will not use this terminology since it might be confused with Euler vector field in Gromov-Witten theory).
Abrams showed that the invertibility of quantum multiplication of $\Delta$ is equivalent to the semi-simplicity of the small quantum cohomology. In \cite{BuP}, Buch and Pandharipande proved that all Tevelev degrees in Gromov-Witten theory are encoded in the
quantum product of the point class and $\Delta^{*k}$, $k \geq 0$.

Note that $\Delta$ is also well defined in the big quantum cohomology which depends on all genus-0 primary Gromov-Witten invariants. In this case, $\Delta$ is considered as a vector field on $H^*(X)$. This vector field was introduced
in \cite{L} and it was called the quantum volume element.
In \cite{L1}, it was showed that the genus-1 Virasoro conjecture for Gromov-Witten invariants can be reduced
to a single equation $E \Psi = 0$ where $\Psi$ is a function on $H^*(X)$ depending on both genus-1 and genus-0 Gromov-Witten invariants, $E$ is the Euler vector field on $H^*(X)$ whose restriction to the origin
is the first Chern class of $X$. Note that $E$ has nothing to do with $\Delta$ and the confusion of terminology is due to  historic reasons. It was proved in \cite{L} that $\Delta \Psi =0$. This result was generalized in \cite{Lin} to the form
$(\Delta*V) \Psi = 0$ for all vector field $V$ on $H^*(X)$.  In particular, these results show that
if $E$ lies in the following space
\begin{equation} \label{eqn:F}
\mathfrak{F} := {\rm Span}\{ \Delta^{*k} \mid k \geq 0 \},
\end{equation}
then the genus-1 Virasoro conjecture holds.

For the above reasons, it is interesting to investigate how large $\mathfrak{F}$ is. In particular,
the dimension of $\mathfrak{F}$ is a numerical invariant for the symplectic structure of $X$.
In this paper, we will give an upper bound for the dimension of $\mathfrak{F}$ for
a class of compact symplectic manifold $X$ with $H_2(X, \mathbb{Z}) \cong \mathbb{Z}$ and $H^{\rm odd}(X)=0$.
For such manifolds, the quantum product defines a ring structure on $QH^*(X)_q = H^*(X) \otimes \mathbb{C}[q,q^{-1}]$
where $q$ is a formal parameter which can be assigned a (complex) degree $\tau=\deg(q)$ such that
the quantum multiplication preserves the total degree.
Define $\mathfrak{F}_q := {\rm Span}_{\mathbb{C}[q,q^{-1}]} \{\Delta^{*k} \mid k \geq 0\}$.
For Fano varieties, we can recover $\mathfrak{F}$ from $\mathfrak{F}_q $ by setting $q=1$.
We have
\begin{theorem}\label{thm: main thm 2}
Let  $X$ be any compact symplectic manifold with $H_2(X, \mathbb{Z}) \cong \mathbb{Z}$ and $H^{\rm odd}(X)=0$.
Assume either $\Delta$ or the class of a point $[pt]$ is invertible in $QH^*(X)_q$. Then
the rank of $\mathfrak{F}_q$ over $\mathbb{C}[q,q^{-1}]$ has the following upper bound:
\[
 \mathrm{rank} (\mathfrak{F}_q) \leq \frac{\tau}{{\rm gcd}(\tau, \dim(X))}
                                \sum_{i=0}^{\lfloor \dim(X)/\tau \rfloor} \dim H^{2i \tau}(X),
\]
where $\dim(X)$ is a half of the real dimension of $X$.
Moreover, for $Gr(2,n)$, this inequality turns out to be an equality  and we obtain a precise formula for
$\dim \mathfrak{F}$.
\end{theorem}
\noindent
Theorem \ref{thm: main thm 2} is a combination of Theorem \ref{thm:DimBoundX} and Theorem \ref{thm: case of Gr(2,n)}.
To get more information for $\mathfrak{F}_q$,  we will give a free submodule of $QH^*(X)_q$ which contains $\mathfrak{F}_q$ and whose rank is equal to the right hand side of inequality in Theorem \ref{thm: main thm 2}
(see Theorem \ref{thm: subspace spanned by Delta^i}). We will prove in Theorem \ref{thm: case of Gr(2,n)}
that this submodule is equal to $\mathfrak{F}_q$ for $Gr(2,n)$.
Hence we have a precise description of $\mathfrak{F}_q$ for $Gr(2,n)$.
From this result and the result in Section \ref{sec:Pn}, we have $\mathfrak{F} = H^*(X)$ if $X$ is a projective space or $Gr(2, n)$ with $n$ odd.
For other Grassmannians, it is more complicated to describe $\mathfrak{F}$. We will give an upper
bound for the dimension of $\mathfrak{F}$ in terms of numbers of restricted partitions in
Theorem \ref{thm: estimation of dimension of Span(Delta^i)}.
We will also prove that asymptotically, $\dim(\mathfrak{F}) / \dim H^*(Gr(k,n))$ is bounded from
 above by $1/{\rm gcd}(n, k^2)$
(see Corollary \ref{cor:SymptBound}).
This bound shows that for many $Gr(k,n)$, the dimension of $\mathfrak{F}$ could be much smaller than the dimension of
$H^*(Gr(k,n))$.
For Fano complete intersections, the dimension of  $\mathfrak{F}$ can be computed precisely. The answers will be given in
Remark \ref{rem:M<LR-1}, Propositions \ref{prop: m=r+L hprim=0 finite states} and \ref{prop:dimFCI}.

In the study of complexity, eigenvalues of quantum multiplication by $\Delta$ play an important role. In general, such eigenvalues could be very complicated.
In \cite{BuP}, Buch and Pandharipande proved that for (co)minuscule varieties, quantum multiplication by $\Delta/[pt]$ is always diagonalizable over $\mathbb{R}$. In this paper, we will prove the following positivity result which is needed in the proofs of Theorems \ref{thm: main thm 1} and \ref{thm: main thm 2}.
\begin{theorem} \label{thm:eigenvalue>0}
For any (co)minuscule homogeneous variety,
all eigenvalues of quantum multiplication by $\Delta/[pt]$ are positive real numbers.
\end{theorem}
\noindent
This theorem is a corollary of Proposition~\ref{prop: main result that A is symmetric}.

This paper is organized as follows: In Section \ref{sec: preliminary}, we review definitions and basic properties about complexity of 2D TQFT and quantum cohomology. In Section \ref{sec: examples of Pn and Qn},  we
give explicit description for states with finite complexity and the sets $\mathfrak{S}_\infty$ and $\mathfrak{F}$ for projective spaces and quadrics. In Section \ref{sec: cominuscule}, we estimate the size of $\mathfrak{S}_\infty$
and prove Theorem \ref{thm:eigenvalue>0} for minuscule and cominuscule spaces.
In Section \ref{section: Periodicity of action of Delta}, we estimate the dimension of $\mathfrak{F}$
for compact symplectic manifolds with single quantum parameter.
In Section \ref{sec: quantum cohomology}, we give an explicit formula representing $\Delta$ as
a linear combination of Schubert classes for Grassmannians and  prove the sharpness of the bound in
Theorem \ref{thm: main thm 2}
for ${\rm Gr}(2,n)$. In Section \ref{sec: fano complete intersection}, we study complexity for quantum cohomology of Fano complete intersections and finish the proof of Theorem \ref{thm: main thm 1}.
In the appendix, we will prove a result about matrices with real eigenvalues which can be used to estimate the size of
$\mathfrak{S}_\infty$.

{\it Acknowledgments}: The authors would like to thank Longting Wu for bringing the reference \cite{Ce} to our attention.

\section{Preliminaries}\label{sec: preliminary}

\subsection{Complexity in 2D TQFT}    \label{subsec 2D TQFT}

In this subsection, we review the definition of circuit complexity for 2D TQFT as given in \cite{CFS}.
We refer to the book \cite{K} for definitions and basic properties of Frobenius algebra and 2D TQFT.

In this paper, a Frobenius algebra always means a symmetric Frobenius algebra over $\mathbb{C}$.
There are several equivalent definitions for
Frobenius algebra. For convenience, we will use the following definition (see Sections 2.2.7 and 2.2.9 in \cite{K}):
A {\it Frobenius algebra} is a finite dimensional commutative associative algebra $\mathcal{H}$  with a unit $\mathbf{1} \in \mathcal{H}$ and a symmetric nondegenerate pairing $g: \mathcal{H} \times \mathcal{H} \longrightarrow \mathbb{C}$ such that
\begin{equation} \label{eqn:FrobCond}
 g(x \cdot y, \, z) \,= \, g(x, \, y \cdot z)
\end{equation}
for all $x, y, z \in \mathcal{H}$, where $x \cdot y$ is the product of $x$ and $y$.

Let 2\textbf{Cob} be the monoidal category whose objects are disjoint unions
of circles and morphisms are compact oriented 2-dimensional surfaces with boundaries. Each such surface is an oriented cobordism
between its incoming boundary and outgoing boundary.
Let $\textbf{Vect}_\mathbb{C}$ be the monoidal category whose objects are finite dimensional vector spaces over $\mathbb{C}$
and morphisms are linear maps between vector spaces.
A {\it 2D TQFT} is a symmetric monoidal functor $\Phi$ from 2\textbf{Cob} to $\textbf{Vect}_\mathbb{C}$.
The functor $\Phi$ induces a Frobenius algebra structure on the vector space $\mathcal{H} = \Phi (S^1)$.
Conversely, given any Frobenius algebra, one can construct a 2D TQFT  in a canonical way
(see Theorem 3.3.2 in \cite{K}).

Let $\Sigma_1$ be the genus-1 surface whose boundary consists of one incoming circle and one outgoing circle:

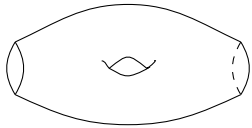
\begin{center}
    \begin{tikzpicture}[scale=0.5]
    \draw(-3,0.7) to[out=25, in= 180] (0,1.7) to[out=0, in=155] (3,0.7);
    \draw(-3,0.7) .. controls (-3.3,0.2) and (-3.3,-0.2) .. (-3,-0.7);
    \draw(-3,0.7) .. controls (-2.7,0.2) and (-2.7,-0.2) .. (-3,-0.7);
    \draw(3,0.7) .. controls (3.3,0.2) and (3.3,-0.2) .. (3,-0.7);
    \draw[dashed](3,0.7) .. controls (2.7,0.2) and (2.7,-0.2) .. (3,-0.7);
    \draw(-3,-0.7) to[out=-25, in= 180] (0,-1.5) to[out=0, in=-155] (3,-0.7);
    \draw(-0.5,0) .. controls (0,0.4) .. (0.5,0);
    \draw(-0.7,0.2) to[out=-25 in= 155] (-0.5,0) to[out=-25, in=180] (0,-0.2) to[out=0, in=25] (0.5,0) to[out=25, in=25] (0.7,0.2);
 \end{tikzpicture}
\end{center}

\noindent
Then $\Phi(\Sigma_1)$ is a linear operator on $\mathcal{H}$ which is called the {\it handle operator}.
More generally, for any compact oriented genus-$h$ surface $\Sigma$ whose boundary consists of one incoming circle and one outgoing circle,
$\Phi(\Sigma)$ defines a linear operator on  $\mathcal{H}$. Topologically $\Sigma$ can always be
obtained by gluing $h$ copies of $\Sigma_1$. Hence $\Phi(\Sigma)$ must be the composition of $h$ copies of the handle
operator, i.e.  $\Phi(\Sigma) = \Phi(\Sigma_1)^h$. This justifies why the handle operator is used as an elementary
gate in the definition of circuit complexity in 2D TQFT.

If we choose a basis $\{ e_1, \ldots, e_n \}$ for the vector space $\mathcal{H}$ and let $g_{ij}=g(e_i, e_j)$, then
the handle operator $\Phi(\Sigma_1)$ is equal to the multiplication by the {\it handle element}
$\Delta_{\mathcal H} \in {\mathcal H}$ which is defined by
\begin{equation} \label{eq:DeltaH}
 \Delta_{\mathcal H} := \sum_{i, j=1}^n g^{ij} e_i \cdot e_j,
\end{equation}
where $g^{ij}$ are entries for the inverse matrix of $(g_{ij})_{n \times n}$ (see Exercise 5 on page 128 in \cite{K}).
Let $\check{e}_i := \sum_{j=1}^n g^{ij} e_j$, then $g(\check{e}_i, \, e_j) = \delta_{ij}$. So $\{\check{e}_i \mid i=1, \ldots, n\}$
is a basis of ${\mathcal H}$ which is dual to the basis $\{e_i \mid i=1, \ldots, n\}$. We also have
\begin{equation} \label{eq:DeltaHdual}
 \Delta_{\mathcal H} := \sum_{i=1}^n  e_i \cdot \check{e}_i .
\end{equation}

According to \cite{CFS}, a (1-party) {\it state} is an element in $\mathbb{P} {\mathcal H}$, which is the projectivization of
${\mathcal H}$. For $x \in {\mathcal H}$, its image in $\mathbb{P} {\mathcal H}$ under the natural projection map is denoted
by $[x]$. Fix a reference state $[x_0]$. Then the (circuit) {\it complexity}
of $[x]$ is defined to be the smallest non-negative integer $k$ such that $[x]=[\Delta_{\mathcal H}^k \cdot x_0]$
where $\Delta_{\mathcal H}^k$ is the $k$-th power of $\Delta_{\mathcal H}$. If such $k$ does not exist, then we define the complexity of $[x]$ to be infinity. Note that the set of states with finite complexity is always countable.
So most states do not have finite complexity. Fix a metric on the projective space $\mathbb{P} {\mathcal H}$.
For any $\epsilon > 0$,  the {\it approximate complexity} of $[x]$ with {\it tolerance} $\epsilon$ is defined
to be the smallest non-negative integer $k$ such that the distance between $[x]$ and $[\Delta_{\mathcal H}^k \cdot x_0]$ is less than or equal to $\epsilon$. Similar to equation \eqref{eqn:SInfty}, we can define $\mathfrak{S}_{\infty}({\mathcal H})$
to be the set of all states which have infinite complexity but have finite approximate complexity with arbitrarily
small tolerance. Then a state lies in $\mathfrak{S}_{\infty}({\mathcal H})$ if and only if it has infinite complexity and
lies in the closure of the set of all states with finite complexity. Hence $\mathfrak{S}_{\infty}({\mathcal H})$ only depends on the topology of $\mathbb{P} {\mathcal H}$ and does not depend on the metric used in defining approximate complexity.
As discussed in Section 4.2 in \cite{CFS}, $\mathfrak{S}_{\infty}({\mathcal H})$ could be very large.  It is possible that
for some  2D TQFT whose corresponding Frobenius algebra is semisimple, $\mathfrak{S}_{\infty}({\mathcal H})$ may contain a torus of positive dimension.

\subsection{Quantum cohomology} \label{subsec:QCoh}

We refer to \cite{RT} and \cite{LiT} for basic definitions of Gromov-Witten invariants and quantum cohomology.
For any compact symplectic manifold $X$ and
$A \in H_{2}(X, {\mathbb Z})$,
the $k$-point degree-$A$ genus-0 {\it Gromov-Witten invariant}
of $\gamma_1, \ldots, \gamma_k \in H^*(X)$
is defined by
\[
 \left<  \gamma_1, \ldots, \gamma_k \right>_{A}
\, := \, 	\int_{\left[\overline{\cal M}_{0, k}(X, A) \right]^{\rm virt}}
	\prod_{i=1}^k {\rm ev}_{i}^{*}(\gamma_{i}),
\]
where $\left[\overline{\cal M}_{0, k}(X, A) \right]^{\rm virt}$
 is the virtual fundamental class on the moduli space of degree $A$ stable maps from genus-$0$ curves
with $k$ marked points to $X$,
and ${\rm ev}_{i}$ is the evaluation
map from the moduli space $\overline{\cal M}_{0, k}(X, A) $
to $X$ defined by evaluating each stable map at the $i$-th marked point.
While the big quantum cohomology depends on all genus-0 invariants, in this paper we mainly consider small
quantum cohomology which only needs 3-point genus-0 invariants.

Choose a basis $\{e_i \mid i=1, \ldots N\}$ of $H^*(X)$.
Entries of the matrix for the {\it Poincar\'e pairing} $g$ is given by
\[ g_{ij} = g(e_i, e_j) = \int_X e_i \cdot e_j \]
where ``$\cdot$'' is the product in the ordinary cohomology ring.
Let $g^{ij}$ be the entries for the inverse matrix of $(g_{ij})_{N \times N}$.
The (small) quantum
product of $\gamma_1, \gamma_2 \in H^*(X)$ is defined by
\begin{equation} \label{eqn:qprod}
 \gamma_1 * \gamma_2 := \sum_{i,j=1}^N \,\, \sum_{A \in H_{2}(X, {\mathbb Z})} q^{A}  \left< \gamma_1, \gamma_2, e_i \right>_{A} g^{ij} e_j,
\end{equation}
where $q^{A}$ belongs to the Novikov ring (i.e. the multiplicative ring
spanned by monomials $q^{A} = q_{1}^{a_{1}} \cdots q_{r}^{a_{r}}$
over the ring of rational numbers, where
$q=(q_{1}, \ldots, q_{r})$ are formal variables corresponding to a fixed basis of $H_{2}(X, {\mathbb Z})$ and
$a_1, \ldots, a_r$ are coefficients of $A$ written as a linear combination of elements in this basis).

If $q=0$, the right hand side of equation \eqref{eqn:qprod} is just $\gamma_1 \cdot \gamma_2$. Hence the quantum product
is a deformation of the product in the ordinary cohomology ring. Let $\mathbf{1} \in H^*(X)$ be the identity in the ordinary cohomology ring. It turns out that the quantum product is associative super-commutative with identity equal to  $\mathbf{1}$. If $X$ does not have nontrivial cohomology classes with odd degree (this is the case for (co)minuscule homogeneous varieties) or if restricted to the space of cohomology classes with even degree, then the quantum product is also commutative.
Moreover
\begin{equation} \label{eqn:CompQPpairing}
g(\gamma_1 * \gamma_2, \, \gamma_3) = g(\gamma_1, \, \gamma_2 * \gamma_3)
\end{equation}
for all $\gamma_1, \gamma_2, \, \gamma_3 \in H^*(X)$.
Hence the quantum product and the Poincar\'e pairing define a Frobenius algebra structure on
\[ QH^*(X) := H^*(X) \bigotimes \mathbb{C}[q]\]
 over the ring $\mathbb{C}[q]$.

The appearance of $q=(q_1, \ldots, q_r)$ in the definition of quantum product is due to a technical reason: We need to avoid the convergence issue for the summation over $A \in H_2(X, \mathbb{Z})$. A basic property of Gromov-Witten invariants is that
if $\left< \gamma_1, \gamma_2, \gamma_3 \right>_A \neq  0 $ and $\gamma_i$ are homogeneous, then we must have
\begin{equation} \label{eqn:Select}
\deg_{\mathbb{R}}(\gamma_1) + \deg_{\mathbb{R}}(\gamma_2) + \deg_{\mathbb{R}}(\gamma_3) = \dim_{\mathbb{R}}(X) + 2 c_1(X)(A),
\end{equation}
where $c_1(X)$ is the first Chern class of $X$.
In this paper, $\dim_{\mathbb{R}}$ and $\deg_{\mathbb{R}}$ stand for real dimension and real degree. We will reserve
the notations $\dim := \frac{1}{2} \dim_{\mathbb{R}}$ and $\deg := \frac{1}{2} \deg_{\mathbb{R}}$ for complex dimension and complex degree which are more convenient to use if $X$ is a K\"ahler manifold. We can assign a (complex) degree to each $q_i$
such that $\deg(x*y)=\deg(x)+\deg(y)$ for all homogeneous $x, y \in H^*(X)$.
In fact $\deg(q_i)$ is just the value of $c_1(X)$ on the basis element of $H_2(X, \mathbb{Z})$ corresponding to $q_i$.
If there is only one quantum parameter $q$, this is the case when $H_2(X, \mathbb{Z}) \cong \mathbb{Z}$, then
we will set $\tau := \deg(q)$ in this paper.

If $X$ is a Fano variety (most examples in this paper satisfy this condition), then there are only finitely many $A$
such that $\left< \gamma_1, \gamma_2, \gamma_3 \right>_A \neq  0 $ for fixed $\gamma_i$. In this case, the right hand side
of equation~\eqref{eqn:qprod} is well defined if we set all formal parameters $q_i=1$. The quantum product and Poincar\'e pairing then define a Frobenius algebra structure on $H^*(X)$ over $\mathbb{C}$, which in turn defines a 2D TQFT.
The handle element for this 2D TQFT is precisely the element $\Delta$ defined by equation \eqref{eqn:Delta}. Hence the complexity defined in \cite{CFS} coincides with the definition given in the introduction.

If $X$ has non-trivial cohomology classes of odd degree (for example, for many Fano complete intersections), then the quantum cohomology of $X$ is only super commutative. In this case, the definition of complexity given in the introduction still makes sense.

\section{Projective spaces and quadrics}\label{sec: examples of Pn and Qn}

Projective spaces $\mathbb{P}^n$ and quadrics $\mathbb{Q}^n$ are special cases of (co)minuscule homogeneous varieties (see Section \ref{sec: cominuscule}).
The small quantum cohomology of these spaces are relatively simple. In this section, we
give an explicit description for ${\mathfrak S}_\infty$, $\mathfrak{F}$, and the set of states with finite complexity for these spaces. We will prove that ${\mathfrak S}_\infty$ is either empty or consists of a single element.
For projective spaces, $\mathfrak{F}$ is equal to the space of cohomology classes.
For quadrics, $\mathfrak{F}$ is a 2-dimensional subspace inside the space of cohomology classes whose dimension could be
arbitrarily large.

\subsection{Projective spaces}
\label{sec:Pn}

The small quantum cohomology ring of $\mathbb{P}^n$ has a very simple description (see, for example, Section 10 of \cite{FuPan}):
\[
QH^*(\mathbb{P}^n)=\mathbb{C}[H,q]/(H^{*(n+1)}-q),
\]
\noindent where $H$ is the hyperplane class and $H^{*i}$ is  the $i$-th quantum power of $H$.
A basis for the vector space $H^*(\mathbb{P}^n)$ is given by $\{\mathbf{1}=H^0, H, ...,H^n\}$,
were $H^i$ is $i$-th power of $H$ with respect to the ordinary cup product.
We have
\[ H^{*(n+1)} = q \mathbf{1} {\rm \,\,\, and \,\,\, } H^{*i}=H^i {\rm \,\,\, for \,\,\,} 0 \leq i \leq n. \]
Since the Poincar\'e pairing on $H^*(X)$ is given by
\[
g(H^i, H^j)=\delta_{i+j,n},
\]
the handle element $\Delta$ is given by
\[
\Delta=\sum_{i=0}^n H^i * H^{n-i}=(n+1)H^n.
\]
Hence
\[
\Delta * H^i=
\begin{cases}
    (n+1)H^n, & i=0, \\
    (n+1)q H^{i-1}, & 1\leq i\leq n,
\end{cases}
\]
and
\[ \Delta^{*k} = (n+1)^k q^{k-1} H^{n+1-k} {\rm \,\,\, for \,\,\,} 1 \leq k \leq n+1. \]

To study complexity of quantum cohomology, we set $q=1$. Since the action of $\Delta$ on $\mathbb{P}H^*(\mathbb{P}^n)$
is periodic with period $n+1$, there are only $n+1$ states with finite complexity and $\mathfrak{S}_\infty$ is an empty set.

If we choose the reference state $S_0=[\mathbf{1}]$, then the set of states with finite complexity is $\{[\mathbf{1}],[H],...,[H^n]\}$ and the complexity of $[H^i]$ is $n+1-i$ for $1 \leq i \leq n$.
Moreover, $\mathfrak{F}=H^*(\mathbb{P}^n)$ where $\mathfrak{F}$ is defined by equation \eqref{eqn:F}.

\subsection{Quadrics $\mathbb{Q}^{r}$}
\label{subsect: even quadrics}

We first recall some basic properties of the quantum cohomology of quadrics given in Section 4.1 of  \cite{CMP2}.
A basis of the vector space $H^*(\mathbb{Q}^{r})$
is given by the set of Schubert classes.
For each $0 \leq i \leq \lfloor (r-1)/2 \rfloor$, there are Schubert classes $\sigma_i$ and $\sigma_{r-i}$
with (complex) degree $i$ and $r-i$ respectively.
If $r=2m$ is even, then there are two more Schubert classes $\sigma_m^+$ and $\sigma_m^-$ of degree $m$.
The identity of the cohomology ring $H^*(\mathbb{Q}^{r})$ is $\mathbf{1}=\sigma_0$ and the class of point is $\sigma_r$.
Let $H$ be the hyperplane class on $\mathbb{Q}^{r}$. Then
\[
H^k=\left\{ \begin{array}{ll} \sigma_k & {\rm \,\,\, if \,\,\,} 0 \leq k \leq \lfloor (r-1)/2 \rfloor, \\
                              \sigma_m^+ + \sigma_m^-  & {\rm \,\,\, if \,\,\,} r=2m {\rm \,\,\, is \,\,\, even \,\,\, and \,\,\,} k=m, \\
                              2\sigma_k & {\rm \,\,\, if \,\,\,} \lceil (r+1)/2 \rceil \leq k \leq r.
            \end{array} \right.
\]
If we replace ordinary cup product by quantum product, these relations are still valid except in the top degree. More precisely, we have
\[  H^{*k} = H^k {\rm \,\,\, for \,\,\,} 0 \leq k < r, {\rm \,\,\, and \,\,\,}
    H^{*r} = 2 \sigma_{r}+2q \mathbf{1}, \,\,\,  \sigma_r * H = qH.
\]
The last two formulas imply that
\begin{equation}  \label{eqn:pt*Hi}
\sigma_r * H^{*i} = q H^{*i}  \hspace{20pt} {\rm for \,\,\, all \,\,\,} i \geq 1
\end{equation}
and (as observed in \cite{BuP})
\begin{equation} \label{eqn:pt*2}
\sigma_r * \sigma_r = q^2 \mathbf{1}.
\end{equation}
If $r=2m$ is even, we also have
\begin{align*}
    & \sigma_{r}=\sigma_m^+*\sigma_m^-, \,\,\, \text{ }\sigma_m^+*\sigma_m^+=\sigma_m^-*\sigma_m^-=q \mathbf{1}.
\end{align*}
Consequently, we have
\begin{equation} \label{eqn:pt*+-}
 \sigma_r * \sigma_m^+ = q \sigma_m^-,  \hspace{20pt} \sigma_r * \sigma_m^- = q \sigma_m^+.
\end{equation}

The following formula for the handle element $\Delta$ of quadrics was computed in Section 3.5 of \cite{BuP}
(where it was called the quantum Euler class):
\begin{equation} \label{eqn:DeltaQ2m}
\Delta= (r+\delta) \sigma_r + (r-\delta) q \mathbf{1},
\end{equation}
where $\delta = 1$ if $r$ is odd and $\delta = 2$ if $r$ is even.
Define
 \[ V_1 := {\rm Span}_{\mathbb{C}[q]} \{\sigma_{i}, \sigma_{r-i} \mid 1 \leq i \leq \lfloor (r-1)/2 \rfloor \}, \hspace{20pt}
    V_2 := {\rm Span}_{\mathbb{C}[q]} \{\mathbf{1}, \sigma_{r} \}.
    \]
If $r=2m$ is even, we also define
\[  V_3 := {\rm Span}_{\mathbb{C}[q]} \{ \sigma_m^+, \sigma_m^- \}. \]
Then $QH^*(\mathbb{Q}^{r}) = V_1 \bigoplus V_2$ for $r$ odd and
$QH^*(\mathbb{Q}^{r}) = V_1 \bigoplus V_2 \bigoplus V_3$ for $r$ even.
\begin{proposition} \label{prop:quadrics}
    The action of $\Delta * $  preserves $V_i$ for each $i$.
    On $V_1$, this action is the scalar multiplication by $2rq$.
    On $V_2$, this action has two  distinct eigenvalues $2rq$ and $-2 \delta q$.
    If $r$ is even, this action on $V_3$ also has two  distinct eigenvalues $2rq$ and $-2 \delta q$.
    Moreover, $\mathfrak{F} \otimes \mathbb{C}[q] = V_2 $ where $\mathfrak{F}$ is defined by equation~\eqref{eqn:F}.
\end{proposition}

\begin{proof}
By equations \eqref{eqn:pt*Hi} and \eqref{eqn:DeltaQ2m},
\begin{equation} \label{eqn:Delta*Hi}
\Delta * H^{*i} = 2rq H^{*i} \hspace{20pt} {\rm for \,\,\, all \,\,\,}  i \geq 1.
\end{equation}
Note that
 $\{ H^{*i}, H^{*(r-i)} \mid 1 \leq i \leq \lfloor (r-1)/2 \rfloor \}$ is also a basis of $V_1$.
This shows that the action by $\Delta *$ on $V_1$ is just the scalar multiplication by $2rq$.
In particular, this action preserves $V_1$.

Equations \eqref{eqn:pt*2} and \eqref{eqn:DeltaQ2m} imply that the action by $\Delta *$
preserves $V_2$. This implies $\Delta^{*k} \in V_2$ for all $k \geq 0$ since $\mathbf{1} \in V_2$.
Consequently $\mathfrak{F} \otimes \mathbb{C}[q] = V_2$.
If $r$ is even, equations \eqref{eqn:pt*+-} and \eqref{eqn:DeltaQ2m} imply that the action by $\Delta *$
preserves $V_3$.

By equation~\eqref{eqn:Delta*Hi}, $q \mathbf{1} + \sigma_r = \frac{1}{2} H^{*r} \in V_2$
and $\sigma_m^+ + \sigma_m^- = H^{*m} \in V_3$ (if $r=2m$ is even) are
eigenvectors of $\Delta *$ with eigenvalue $2rq$.
Using equations \eqref{eqn:pt*2} and \eqref{eqn:pt*+-}, one can check that
$q \mathbf{1} - \sigma_r \in V_2$ and
 $\sigma_m^+ - \sigma_m^-  \in V_3$ (if $r=2m$ is even) are
eigenvectors of $\Delta *$ with eigenvalue $- 2 \delta q$.
The proposition is thus proved.
\end{proof}

We now consider complexity of quantum cohomology and set $q=1$.
By Proposition~\ref{prop:quadrics}, $H^*(X)=E_1 \bigoplus E_2$
where $E_1$ and $E_2$ are the eigenspaces of $\Delta *$ with eigenvalues $2r$ and $-2\delta$ respectively.
Fix a reference state $S_0 \in \mathbb{P}H^*(X)$, we can write
$S_0=[x_1 + x_2]$ with $x_1 \in E_1$ and $x_2 \in E_2$.
For example, if $S_0=[\mathbf{1}]$, we can take $x_1=\mathbf{1}+\sigma_r$ and $x_2=\mathbf{1}- \sigma_r$.
We have
\[ \Delta^{*k} * S_0 = [r^k x_1 +(-\delta)^{k} x_2] \in \mathbb{P}H^*(\mathbb{Q}^{r}) \]
 for all $k \geq 0$.

If $x_1 =0$ or $x_2=0$, then $\Delta^{*k} * S_0 = S_0$ for all $k \geq 0$. So
only $S_0$ has finite complexity and $\mathfrak{S}_\infty$ is an empty set.

Assume $x_1 \neq 0$ and $x_2 \neq 0$.
If $r=\delta$, then for all $k \geq 0$,
\[ \Delta^{*2k} * S_0 = S_0 {\,\,\, \rm and \,\,\, }
\Delta^{*(2k+1)} * S_0=[x_1 - x_2] \in \mathbb{P}H^*(\mathbb{Q}^{r}).
\]
So the set of states with finite complexity is $\{S_0, [x_1-x_2]\}$ and $\mathfrak{S}_\infty$ is an empty set.

If $r > \delta$, then
$\Delta^{*k} * S_0 = [x_1+(-\frac{\delta}{r})^k x_2] \in \mathbb{P}H^*(\mathbb{Q}^{r})$ which converges to $[x_1] $
as $k \rightarrow \infty$. So there are infinitely many states with finite complexity.
The set $\mathfrak{S}_\infty$ consists of only one element $[x_1]$.

\section{(Co)minuscule homogeneous varieties}\label{sec: cominuscule}

In this section, we will study handle elements $\Delta$ for (co)minuscule homogeneous varieties $X$, estimate the size of $\mathfrak{S}_\infty$, and prove the positivity of eigenvalues for quantum multiplication by $\Delta/[pt]$.

\subsection{Quantum cohomology for (co)minuscule homogeneous varieties}

Any (co)minuscule homogeneous variety can be written as $X=G/P$ where $G$ is a semisimple complex algebraic group
and $P$ is a maximal parabolic subgroup of $G$. The set of simple roots of $P$ can be obtained from the set of simple roots
of $G$ by deleting one element $\alpha$. $X$ is  {\it cominuscule} if the coefficient of $\alpha$ in the highest root of $G$ is 1.
Let $\omega$ be the fundamental weight of $G$ corresponding to $\alpha$.
$X$ is {\it minuscule} if $\frac{2 (\omega, \beta)}{(\beta,\beta)} \leq 1$
for all positive root $\beta$ of $G$, where $(\cdot, \cdot)$ is an inner product invariant under the Weyl group $W$ of $G$.
Minuscule and cominuscule homogeneous varieties have been classified. They include
Grassmannians, quadrics, Lagrangian Grassmannians, orthogonal Grassmannians, Cayley plane, and Freudenthal variety (cf. \cite{BL}).
A uniform study of the small quantum cohomology of (co)minuscule homogeneous varieties is given by Chaput, Manivel, and Perrin. Below we recall some basic facts from \cite{CMP2} which will be needed in our calculations.

A basis of vector space $H^*(X)$ is given by Schubert classes $\{ \sigma_w \mid w \in W_X \}$
where $W_X = W/W_P$ and  $W_P$ is the Weyl group of $P$ realized as a subgroup of $W$.
There is an involution $p: W_X \longrightarrow W_X$ (defined using multiplication by the longest element of $W$) such that
the Poincar\'{e} pairing on $H^*(X)$ is given by
\begin{equation} \label{eqn:PDmin}
 \langle \sigma_u,\sigma_v\rangle = g(\sigma_u,\sigma_v) =\delta_{up(v)} {\rm \,\,\, for \,\,\, all \,\,\,} u,v \in W_X,
\end{equation}
where $\delta_{uv}$ equals 1 if $u=v$ and 0 otherwise.

The quantum product of Schubert classes can be written as
\begin{equation}\label{formula: multiply schubert classes}
    \sigma_u*\sigma_v=\sum_{w\in W_X} \sum_{d \geq 0} C_{uv}^{d,w} q^d\sigma_w,
\end{equation}
where $C_{uv}^{d,w} := \langle\sigma_u, \sigma_v, \sigma_{p(w)}\rangle_d$
is the genus 0 three point Gromov-Witten invariants of degree $d$ (see Section \ref{subsec:QCoh}).
In homogeneous space $X$, $C_{uv}^{d,w}$ are always nonnegative integers (see, for example, \cite{FuPan}).
We can assign a (complex) degree $\tau$ to the quantum parameter $q$ such that
the right hand side of equation \eqref{formula: multiply schubert classes} is homogeneous.
In fact, $\tau$ is defined by evaluating the first Chern class of $X$ over the generator of
$H_2(X, \mathbb{Z}) \cong \mathbb{Z}$.
For any $w\in W_X$, let $l(w)$ be the (complex) degree of $\sigma_w$ in $H^*(X)$.
By equation \eqref{eqn:Select}, $C_{uv}^{d,w}$ is non-zero only if
\begin{equation}\label{eq: condition for Cuvdw to be nonzero}
    l(u)+l(v)=l(w)+ d \tau.
\end{equation}
Quantum product $*$ defines ring structures on
\begin{equation}\label{def of QHq}
    QH^*(X) := H^*(X)\otimes\mathbb{C}[q] \hspace{10pt} {\rm and}  \hspace{10pt}
    QH^*(X)_q := H^*(X)\otimes\mathbb{C}[q, q^{-1}].
\end{equation}
Schubert classes $\{\sigma_w \mid w \in W_X \}$ form a basis of
$QH^*(X)_q$ over $\mathbb{C}[q, q^{-1}]$.
A basis of $QH^*(X)_q$ over $\mathbb{C}$ is given by
$\{ q^{d}\sigma_w \mid w\in W_X, d \in \mathbb{Z} \}$.
A crucial property proved by Chaput-Manivel-Perrin is the following {\it strange duality}, which generalizes the same property for
Grassmannians proved by Postnikov in \cite{P}.
\begin{theorem}(Theorem 1.1 in \cite{CMP2}) \label{thm: strange duality}
    Let $X$ be a (co)minuscule homogeneous variety. There exist maps $y: W_X \longrightarrow \mathbb{Q}_{>0}$
    and $\delta: W_X \longrightarrow \mathbb{Z}_{\geq 0}$
    such that the $\mathbb{C}$-linear map
     $\iota: QH^*(X)_q  \longrightarrow QH^*(X)_q$ defined by
     \[
    \iota(q)= y_0 q^{-1}, \hspace{20pt} \iota(\sigma_w)= q^{-\delta(w)}y(w)\sigma_{\iota(w)}
                {\rm \,\,\, \,\,\,  for  \,\,\,} w \in W_X
    \]
is a ring involution on $QH^*(X)_q$, where
$y_0$ is the value of $y$ at the equivalence class of the reflection along the highest root of $G$,
and $\iota(w) \in W_X$ is defined using multiplication by the longest element in $W_P$.
\end{theorem}
In fact the map $w \mapsto \iota(w)$ also defines an involution on $W_X$.
In this paper, we will not need the precise definition of $y(w)$ and $\delta(w)$ which can be found in \cite{CMP2}.
We also do not need the precise definition of the maps $p, \iota: W_X \longrightarrow W_X$ as long as we know that they are involutions.

We will need the following property
\begin{lemma}\label{lem: delta for w and iota w are the same}
    $\delta(w)=\delta(\iota(w))$ and $y(w)y(\iota(w))=y_0^{\delta(w)}$ for all $w\in W_X$.
\end{lemma}
\begin{proof}
Since $\iota$ is an involution, we have
\begin{align*}
       \sigma_w =  \iota(\iota(\sigma_w))&=\iota(q^{-\delta(w)}y(w)\sigma_{\iota(w)})
        = y(w) y_0^{-\delta(w)} q^{\delta(w)-\delta(\iota(w))}y(\iota(w)) \sigma_w.
\end{align*}

\noindent The coefficients of $\sigma_w$ on both sides of this equation must be the same. This implies the
desired equalities.
\end{proof}

Let $[pt]$ be the class of a point, which is equal to the Schubert class associated to
the equivalence class of the longest element of $W$. The following result was also proved by Chaput-Manivel-Perrin:
\begin{theorem}( Theorem 3.3 in \cite{CMP2}) \label{thm mutiply class of pt}
    $[pt]*\sigma_w=q^{\delta(w)}\sigma_{p\iota(w)}$ for all $w\in W_X$.
\end{theorem}

Let $\theta$ be the smallest positive integer (always exists since $p$ and $\iota$ are permutations of the finite set $W_X$) such that
\begin{equation} \label{eqn:theta}
(p\iota)^{\theta}=id.
\end{equation}
The identity $\mathbf{1}$ of cohomology ring $H^*(X)$ is the Schubert class associated to the
equivalence class of the unit element in $W$ (see Section 2.6 of \cite{BL}).
By Theorem \ref{thm mutiply class of pt},  we have
\begin{equation}\label{eqn: certain power of [pt]}
    [pt]^{*\theta} = [pt]^{*\theta}* \mathbf{1} = q^{n(X)} \mathbf{1},
\end{equation}
for some integer $n(X)$.  By counting degrees on both sides of the above equation, we have
$ \theta \dim(X) = n(X) \tau $,
where $\dim(X)=\deg ([pt])$ is the (complex) dimension of $X$ and $\tau = \deg(q)$.
Hence we have
\begin{equation} \label{eqn:n(X)}
n(X) =\frac{\theta}{\tau} \, \dim(X) \in \mathbb{Z}.
\end{equation}

\subsection{Action of handle operator}\label{subsec: action of handle operator}

Since the Poincar\'{e} pairing on (co)minuscule homogeneous variety $X$ is given by equation \eqref{eqn:PDmin},
by equation \eqref{eq:DeltaHdual},
the handle element $\Delta$ for $X$ is given by
\begin{equation}\label{eqn: form of Delta}
    \Delta=\sum_{w\in W_X} \sigma_w * \sigma_{p(w)}.
\end{equation}
By Theorem \ref{thm mutiply class of pt}, the quantum multiplication of $[pt]$ is invertible in $QH^*(X)_q$.
Let $[pt]^{-1} \in QH^*(X)_q$ be the inverse of $[pt]$ and $\Delta/[pt] := \Delta * [pt]^{-1}$.
Using strange duality, Buch and Pandharipande proved that there exists a strange inner product
on $H^*(X, \mathbb{Q}) \otimes \mathbb{R}[q, q^{-1}]$ such that the quantum multiplication by $\Delta/[pt]$ is symmetric
with respect to the strange inner product. In particular, after setting $q=1$, the quantum multiplication by $\Delta/[pt]$ is diagonalizable over $\mathbb{R}$. The definition of the strange inner product involves
both quantum multiplication and the involution $\iota$ (see Section 4 in \cite{BuP}).
In this subsection, we will prove that all eigenvalues of $\Delta/[pt]$ are positive real numbers and thus give
a proof of Theorem~\ref{thm:eigenvalue>0}.
This positivity result will be needed in the proofs of Theorems \ref{thm: main thm for finiteness} and
 \ref{thm: case of Gr(2,n)}.

Recall $\tau=\deg(q)$.
We will construct an explicit  basis of
\[ QH^*(X)_{q, \tau} := H^*(X) \otimes \mathbb{C}[q^{1/\tau}, q^{-1/\tau}] \]
such that under this basis, the matrix $A$
of the quantum multiplication by $\Delta/[pt]$ is a real symmetric matrix with non-negative entries.
Our construction does not rely on results in \cite{BuP}.
We will also give an explicit formula for $A$ and prove that it is positive definite.
This will imply the positivity of all eigenvalues of quantum multiplication by $\Delta/[pt]$.

For any $w\in W_X$, define
\begin{equation} \label{eqn:f(w)}
\tilde{\sigma}_w:=f(w) \sigma_w, \hspace{15pt} {\rm where \,\,\, } \hspace{10pt}
f(w):=y_0^{l(w)/(2\tau)} y(w)^{-1/2} q^{-l(w)/\tau}.
\end{equation}
Then $\{\tilde{\sigma}_w\vert \text{ }w\in W_X\}$ is a basis of $QH^*(X)_{q,\tau}$
over $\mathbb{C}[q^{1/\tau}, q^{-1\tau}]$.
Let
\[ A=(A_{vu})_{u,v \in W_X} \]
 be the matrix of quantum multiplication by $\Delta/[pt]$ with respect to this basis.
Then for all $u \in W_X$,
\begin{equation}\label{eq: definition of matrix A}
    \Delta * \tilde{\sigma}_u = \sum_{v\in W_X} A_{vu} \,\, [pt] * \tilde{\sigma}_v.
\end{equation}
\begin{lemma} \label{lem:matrixA}
For all $u,v \in W_X$,
\begin{align}
A_{vu}  &= \sum_{\substack{w,z\in W_X\\d,d'\geq0}}
                    C_{uw}^{d',z}C_{vw}^{d,z}\frac{y(z)y_0^{\frac{d+d'}{2}}}{(y(u)y(v))^{1/2}y(w)} \,\,\,
                    \in \mathbb{R}_{\geq 0},  \label{eq:exppression for A_uv}
    \end{align}
where $C_{vw}^{d,z}$ are Gromov-Witten invariants defined in equation \eqref{formula: multiply schubert classes}
and $y(w)$ is defined in Theorem \ref{thm:  strange duality}.
In particular, $A$ is a symmetric matrix.
\end{lemma}
\begin{proof}
By equations \eqref{eqn:CompQPpairing} and \eqref{eqn: certain power of [pt]}, we have for all $w, v \in W_X$,
\begin{equation}
\langle [pt] * \sigma_w, \, \, [pt]^{*(\theta-1)} * \sigma_{p(v)}\rangle
= \langle \sigma_w, \,\, [pt]^{*\theta} * \sigma_{p(v)}\rangle
= q^{n(X)} \delta_{wv}.
\end{equation}
Extend the Poincar\'{e} pairing $\langle \cdot, \cdot \rangle$ to
$QH^*(X)_{q, \tau}$ by linearity over $\mathbb{C}[q^{1/\tau}, q^{-1/\tau}]$.
By equation \eqref{eq: definition of matrix A},
 \begin{align*}
        \langle \Delta*\tilde{\sigma}_u,\,\, [pt]^{*(\theta-1)}*\sigma_{p(v)}\rangle
        &=\sum_{w\in W_X}A_{wu}\langle [pt]*\tilde{\sigma}_w, \,\, [pt]^{*(\theta-1)}*\sigma_{p(v)}\rangle\\
        &=\sum_{w\in W_X}q^{n(X)}f(w)A_{wu}\delta_{wv}\\
        &=q^{n(X)}f(v)A_{vu}.
    \end{align*}
Hence $A_{vu}=q^{-n(X)}f(v)^{-1}\langle \Delta*\tilde{\sigma}_u, \,\, [pt]^{*(\theta-1)}*\sigma_{p(v)}\rangle$.
Plugging in the formula for $\Delta$ in equation \eqref{eqn: form of Delta}, we have
    \begin{align}
        A_{vu}&=q^{-n(X)}f(v)^{-1}\sum_{w\in W_X}
            \langle \sigma_w*\sigma_{p(w)}*\tilde{\sigma}_u, \,\, [pt]^{*(\theta-1)}*\sigma_{p(v)}\rangle\nonumber\\
        &=q^{-n(X)}f(u)f(v)^{-1}\sum_{w\in W_X}
            \langle \sigma_u*\sigma_w, \,\, [pt]^{*(\theta-1)}*\sigma_{p(v)}*\sigma_{p(w)}\rangle,
           \label{eqn:App}
    \end{align}
where the second equality follows from equation \eqref{eqn:CompQPpairing}.
To simplify this formula, we need to compute $\sigma_{p(v)}*\sigma_{p(w)}$.

We first find a formula relating  $\sigma_{p(w)}$ to $\sigma_{\iota(w)}$.
Since $\iota$ is an involution, by Theorem \ref{thm mutiply class of pt}
and Lemma \ref{lem: delta for w and iota w are the same}, we have
\[
[pt] *\sigma_{\iota(w)}=q^{\delta(w)} \sigma_{p(w)}.
\]
Multiplying both sides of this equation by $[pt]^{*(\theta-1)}$ and using equation \eqref{eqn: certain power of [pt]},
we have
\begin{equation}\label{eqn: form of sigma_iota(w)}
    \sigma_{\iota(w)}=q^{\delta(w)-n(X)}[pt]^{*(\theta-1)}*\sigma_{p(w)}.
\end{equation}

Applying the ring involution $\iota$ to both sides of
equation \eqref{formula: multiply schubert classes}, we have
\[
q^{-\delta(u)}y(u)\sigma_{\iota(u)}*q^{-\delta(v)}y(v)\sigma_{\iota(v)}=\sum_{\substack{w\in W_X\\d\geq0}}C_{uv}^{d,w}y_0^d q^{-d-\delta(w)}y(w)\sigma_{\iota(w)}.
\]
\noindent Using Equation \eqref{eqn: form of sigma_iota(w)}, one gets

\begin{align*}
    &q^{-n(X)}y(u)[pt]^{*(\theta-1)}*\sigma_{p(u)}*q^{-n(X)}y(v)[pt]^{*(\theta-1)}*\sigma_{p(v)}\\=&\sum_{\substack{w\in W_X\\d\geq0}}C_{uv}^{d,w}y_0^dq^{-d-n(X)}y(w)[pt]^{*(\theta-1)}*\sigma_{p(w)}.
\end{align*}

\noindent Since $[pt]$ is invertible in $QH^*(X)_q$, we can get rid of one copy of $[pt]^{*(\theta-1)}$
from both sides of this equation and obtain
\begin{equation}\label{eq: multiply poincare dual of schubert class}
    [pt]^{*(\theta-1)}*\sigma_{p(u)}*\sigma_{p(v)}=\sum_{\substack{w\in W_X\\d\geq0}}q^{n(X)-d}C_{uv}^{d,w}\frac{y(w)y_0^d}{y(u)y(v)}\sigma_{p(w)}.
\end{equation}
Plugging this formula into the right hand side of equation \eqref{eqn:App}, we obtain
   \begin{align}
        A_{vu} &=f(u)f(v)^{-1}\sum_{w\in W_X}\sum_{\substack{z\in W_X\\d\geq0}}q^{-d}C_{vw}^{d,z}\frac{y(z)y_0^d}{y(v)y(w)}\langle\sigma_u*\sigma_w,\sigma_{p(z)}\rangle\nonumber\\
        &=\sum_{w\in W_X}\sum_{\substack{z\in W_X\\d\geq0}}
          \sum_{\substack{z'\in W_X\\d'\geq0}}
                q^{d'-d+\frac{l(v)-l(u)}{\tau}}C_{uw}^{d',z'}C_{vw}^{d,z}
                \frac{y(z)y_0^{d+\frac{l(u)-l(v)}{2\tau}}}{(y(u)y(v))^{1/2}y(w)} \delta_{z'z},
                  \label{eqn:Alu-lv}
    \end{align}
where the second equality used the definition for $f(w)$ and
equation   \eqref{formula: multiply schubert classes}.

By equation \eqref{eq: condition for Cuvdw to be nonzero},
    $C_{uw}^{d',z} C_{vw}^{d,z}\neq 0$ only if
    \[ \tau d'+l(z)=l(u)+l(w), \hspace{20pt} \tau d+l(z)=l(v)+l(w).\]
     Subtracting these two equations, we obtain
    $l(u)-l(v)=(d'-d)\tau$. Plugging this formula into the right hand side of equation \eqref{eqn:Alu-lv},
    we obtain the formula for $A_{vu}$ in equation \eqref{eq:exppression for A_uv},
    which is clearly symmetric with respect to $u$ and $v$.
    Since the value of $y$ are positive rational numbers and $C_{vw}^{d,z}$ are non-negative integers,
    $A_{vu}$ is a non-negative real number. The lemma is thus proved.
\end{proof}

\begin{proposition}\label{prop: main result that A is symmetric}
    Matrix $A$ is positive definite.
\end{proposition}

\begin{proof}
By Lemma \ref{lem:matrixA}, $A$ is a real symmetric matrix. So it defines a quadratic form
$\Omega$ on $H^*(X, \mathbb{R})$.
For all $x=\sum_{u\in W_X}a_u\sigma_u, y=\sum_{u\in W_X}b_u\sigma_u  \in H^*(X, \mathbb{R})$, we have
\begin{eqnarray*}
\Omega(x,y) &:=& \sum_{u, v \in W_X} a_u b_v A_{uv}
 = \sum_{\substack{u,v,w,z\in W_X\\d,d'\geq0}} a_u b_v C_{uw}^{d',z}C_{vw}^{d,z}\frac{y(z)y_0^{\frac{d+d'}{2}}}{(y(u)y(v))^{1/2}y(w)} \\
 &=& \sum_{w,z\in W_X} L_{w,z}(x) L_{w,z}(y),
\end{eqnarray*}
where
\[L_{w,z}(x):=\sum_{u\in W_X}a_u\sum_{d\geq0} C^{d,z}_{uw} \left( \frac{y(z)y_0^d}{y(u)y(w)} \right)^{\frac{1}{2}} \]
for $w, z \in W_X$ and $x=\sum_{u\in W_X}a_u\sigma_u \in H^*(X, \mathbb{R})$.
In particular
\begin{equation} \label{eqn:Omegaxx}
\Omega(x,x)=\sum_{w,z\in W_X} L_{w,z}(x)^2 \geq 0
\end{equation}  for all $x\in H^*(X, \mathbb{R})$.
Hence all eigenvalues of $A$ are non-negative real numbers.

It was proved in \cite{CMP3} that the small quantum cohomology of $X$ specialized at $q=1$ is semisimple. By Theorem 3.4 in \cite{A}, this is equivalent to
the invertibility of $\Delta$ in the specialized small quantum cohomology. Since $[pt]$ is also invertible, $A$ can not have zero eigenvalue. Hence all eigenvalues of $A$
are positive. Therefore  $A$ is positive definite.
\end{proof}

Since eigenvalues of a linear operator do not depend on the choices of basis, an immediate consequence of
Proposition~\ref{prop: main result that A is symmetric} is that
all eigenvalues of quantum multiplication by $\Delta/[pt]$ are positive real numbers.
This proves Theorem~\ref{thm:eigenvalue>0}.

\begin{remark}
Quantum multiplication
by $\Delta$ could have negative eigenvalues (see for example, the case of quadrics in Section \ref{subsect: even quadrics}).
\end{remark}

\begin{remark}
In the proof of Proposition~\ref{prop: main result that A is symmetric},
we can also prove $A$ is positive definite without using the semisimplicity of the quantum cohomology of $X$.
By Theorem 3.4 in \cite{A}, this also gives a new proof for the semisimplicity of the quantum cohomology of $X$.
As in the proof of Proposition~\ref{prop: main result that A is symmetric}, we only need to show
$\Omega(x, x) \neq 0$ for all non-zero $x \in H^*(X, \mathbb{R})$.

Assume $\Omega(x,x)=0$ for some $x=\sum_{u\in W_X} a_u \sigma_u$ with $a_u \in \mathbb{R}$.
By equation \eqref{eqn:Omegaxx}, we must have
 $L_{w,z}(x)=0$ for all $w,z\in W_X$.
Let $\tilde{x}:=\sum_{u\in W_X}\frac{a_u}{y(u)^{\frac{1}{2}}}\sigma_u$.
Consider the quantum product on $H^*(X, \mathbb{R})$ defined by setting $q=y_0^\frac{1}{2}$.
This is well defined since the summation over $d$ in equation \eqref{formula:  multiply schubert classes} is a finite sum.
Then
\begin{align*}
    \tilde{x}*\sigma_w
    &=\sum_{u\in W_X}\sum_{d\geq0}\sum_{z\in W_X} a_u C^{d,z}_{uw} \left( \frac{y_0^d}{y(u)} \right)^{\frac{1}{2}}\sigma_z
    =\sum_{z\in W_X} \left(\frac{y(w)}{y(z)}\right)^{\frac{1}{2}} L_{w,z}(x)\sigma_z=0.
\end{align*}
Taking $\sigma_w=\mathbf{1}$, we get $\tilde{x}=0$. Hence $x=0$. This shows that $A$ is positive definite.
\end{remark}


%





\subsection{Estimate the size of $\mathfrak{S}_\infty$}
\label{subsec: compute constructible state}

In this subsection, we prove the following result which is a special case of Theorem \ref{thm:  main thm 1}.

\begin{theorem}\label{thm: main thm for finiteness}
    Assume $X$ is a (co)minuscule variety. For any reference state $S_0$, the number of points in $\mathfrak{S}_\infty$ is less than or equal to $\theta$.
\end{theorem}

\begin{proof}
By equation \eqref{eqn:  certain power of [pt]}, $(\Delta/[pt])^{*\theta} = q^{-n(X)} \Delta^{*\theta}$. So,
with respect to the basis $\{ \tilde{\sigma}_w \mid w \in W_X\}$ defined by equation \eqref{eqn:f(w)}, the matrix of
quantum multiplication by $\Delta^{*\theta}$ is $q^{n(X)} A^\theta$,
where $A$ is given by Lemma \eqref{lem:matrixA}.
To study complexity of quantum cohomology, we set $q=1$. Then $\{ \tilde{\sigma}_w \mid w \in W_X\}$ is a basis of $H^*(X)$ and
the matrix of quantum multiplication by $\Delta^{*\theta}$ with respect to this basis is $A^\theta$.
By Proposition~\ref{prop: main result that A is symmetric}, $A^\theta$ is a positive definite real symmetric matrix.
So $H^*(X)$ can be decomposed as a direct sum of eigenspaces $E_1, \cdots, E_k$ of quantum multiplication by
$\Delta^{*\theta}$ with eigenvalues
$\lambda_1 > \cdots > \lambda_k > 0$ respectively.
For any reference state $S_0=[x] \in \mathbb{P}H^*(X)$ where $x \in H^*(X)$,
there exists unique $1 \leq h \leq k$ and $x_i \in E_i$ for $h \leq i \leq k$ such that
\[ x = x_h + x_{h+1} + \cdots + x_k , \hspace{20pt} x_h \neq 0. \]

For any $S \in \mathfrak{S}_\infty$, there exists an infinite sequence $\{n_j \mid j \geq 1\}$ such that
$\Delta^{*n_j} * S_0$ converges to $S$ as $j \rightarrow \infty$.
Since $\theta$ is finite, there exists $0 \leq r \leq \theta - 1$ such that
there exists an infinite subsequence, still denoted by $\{n_j \mid j \geq 1\}$, which is contained
in the set $\{ k \theta + r \mid k \in \mathbb{Z} \}$. For $n_j$ in this subsequence, we can write
$n_j = m_j  \theta + r$ for some $m_j \in \mathbb{Z}$. Note that $\Delta^{*\theta} * x_i=\lambda_i x_i$
for  $h \leq i \leq k$. We have
\begin{align*}
    \Delta^{*n_j} * S_0 & = \Delta^{*r} * \Delta^{*\theta m_j} * S_0
    =\Delta^{*r} * \bigg[ \sum_{h \leq i\leq k} \lambda_i^{m_j} x_i \bigg] \\
    & =\Delta^{*r} * \bigg[ x_h + \sum_{h+1 \leq i\leq k} \left( \frac{\lambda_i}{\lambda_h} \right)^{m_j} x_i \bigg]
        \in \mathbb{P}H^*(X).
\end{align*}
which converges to $\Delta^{*r} * [x_h] = \left[ \Delta^{*r} * x_h \right]$ as $j \rightarrow \infty$ since
$\lambda_i < \lambda_h$ for $i>h$.
Note that $\Delta^{*r} * x_h \neq 0$ since $x_h \neq 0$ and $\Delta$ is invertible. So
$\left[ \Delta^{*r} * x_h \right] \in \mathbb{P}H^*(X)$ is well defined.
Therefore we have $S=\left[ \Delta^{*r} * x_h \right]$.
Consequently $\mathfrak{S}_\infty \subset \{ \left[ \Delta^{*r} * x_h \right] \mid 0 \leq r \leq \theta - 1\}$
which is a finite set. Moreover, $\mathfrak{S}_\infty$ contains at most $\theta$ points.
\end{proof}

\begin{remark}
The same proof shows that
$\mathfrak{S}_\infty$ is equal to the set obtained from $\{ \left[ \Delta^{*r} * x_h \right] \mid 0 \leq r \leq \theta - 1\}$
by removing those states with finite complexity.
\end{remark}

\section{Estimate dimension of $\mathfrak{F}$}
\label{section: Periodicity of action of Delta}

Let $X$ be a compact symplectic manifold with $H_2(X, \mathbb{Z}) \cong \mathbb{Z}$.
We will assume $X$ only has non-trivial cohomology classes of even real degree
(Otherwise, we can also restrict to the space of cohomology classes with even degrees).
In this section, we give a submodule of $QH^*(X)_q = H^*(X) \otimes \mathbb{C}[q,q^{-1}]$ which contains
\begin{equation} \label{eqn:Fq}
 \mathfrak{F}_q := {\rm Span}_{\mathbb{C}[q,q^{-1}]} \{ \Delta^{*k} \in QH^*(X) \mid k \geq 0 \}.
\end{equation}
For Fano varieties, the space $\mathfrak{F}$ defined by equation \eqref{eqn:F} can be obtained from $\mathfrak{F}_q$ by setting $q=1$.
The rank of this submodule gives an upper bound for the dimension of $\mathfrak{F}$.


Let $\tau=\deg(q) = \deg_{\mathbb{R}}(q)/2$ and $\dim(X) = \dim_{\mathbb{R}}(X)/2$.
For all $i \in \mathbb{Z}$, define
\begin{equation} \label{eqn:Ri}
 V_i:= \bigoplus_{k \equiv i ({\rm mod}\,\, \tau)} H^{2k}(X).
\end{equation}
Note that $V_i=V_j$ for $i \equiv j \text{ }(\text{mod} \text{ } \tau)$. We have
\begin{lemma}\label{lem: multiplying Delta between subspaces}
    Quantum multiplication by the handle element $\Delta$ maps $V_i$ into  $V_{i+\text{dim(X)}}\otimes\mathbb{C}[q]$.
\end{lemma}

\begin{proof}
  Since quantum multiplication preserves
 degree (counting both degree of $q$ and degree of cohomology classes),  $\deg(\Delta)= \dim(X)$.
 For any $x \in V_i$, $\deg(\Delta * x)= \dim(X) + \deg(x)$.
 Hence $\Delta * x =  \sum_k C_k q^{n_k} y_k$ with  $y_k \in H^{2k}(X)$,
 $C_k \in \mathbb{C}$, $n_k \in \mathbb{Z}_{\geq 0}$, and
 \[ k = \text{dim}(X) + \deg(x) -\tau n_k \equiv \text{dim}(X) + i \,\,\,(\text{mod}\; \tau). \]
  Hence $\Delta * x \in V_{i+\text{dim(X)}} \otimes\mathbb{C}[q]$.
\end{proof}

Let
\begin{equation} \label{eqn:DX}
D_X:=\text{gcd}(\tau,\text{dim}(X)).
\end{equation}

\begin{theorem}\label{thm: subspace spanned by Delta^i}
For any compact symplectic manifold $X$ with $H_2(X, \mathbb{Z}) \cong \mathbb{Z}$ and $H^{\rm odd}(X)=0$,
    \begin{equation} \label{eqn:SpaceContainF}
    \mathfrak{F}_q \subset \bigoplus_{\substack{ j \equiv 0\;(mod\;D_X) \\ 0 \leq j \leq \tau-1 }}
               V_j\otimes\mathbb{C}[q,q^{-1}].
    \end{equation}
\end{theorem}

\begin{proof}
Let $d = \dim(X)$. First note that $\mathbf{1} \in V_0$. For all integers $k \geq 0$,
by Lemma \ref{lem: multiplying Delta between subspaces},
\[ \Delta^{*k} = \Delta^{*k} * \mathbf{1} \in V_{k d} \otimes \mathbb{C}[q]. \]
There exist integers $m$ and $0 \leq j \leq \tau-1$ such that $kd = m \tau + j$.
Then $\Delta^{*k} \in V_j \otimes \mathbb{C}[q]$ since $j \equiv kd \;(mod\; \tau)$ and $V_{k d}=V_j$. Moreover, we also
have  $j \equiv 0\;(mod\;D_X)$ since both $d, \tau \equiv 0\;(mod\;D_X)$.
Since $\mathfrak{F}_q$ is spanned by $\{ \Delta^{*k} \mid k \geq 0\}$,
this finishes the proof of the theorem.
\end{proof}

This theorem implies that the rank of $\mathfrak{F}_q$ over $\mathbb{C}[q,q^{-1}]$ has the following upper bound:
\begin{equation} \label{eqn:DimBoundX-AnyPt}
 \mathrm{rank} (\mathfrak{F}_q) \leq \sum_{\substack{ j \equiv 0\;(mod\;D_X) \\ 0 \leq j \leq \tau-1 }} \dim (V_j ).
\end{equation}

\begin{lemma}  \label{lem: same of cardinality}
    Assume there exists an invertible homogeneous element $\Phi \in QH^*(X)_q$ with $\deg(\Phi)=\dim(X)$.
    For any $i,j\in\mathbb{Z}$ with $i\equiv j\,(\text{mod}\;D_X)$, we have an isomorphism  between
    $V_i\otimes\mathbb{C}[q,q^{-1}]$ and $V_j \otimes \mathbb{C}[q,q^{-1}]$ as modules over $\mathbb{C}[q,q^{-1}]$.
\end{lemma}

\begin{proof}
    Since $D_X$ divides $\tau$, we can choose integer $M>0$ such that $s:=\frac{M\tau+j-i}{D_X}$ is a positive integer.
    Then $i+sD_X \equiv j \;(\text{mod}\;\tau)$.
    Let $d = \dim(X) = \text{deg}(\Phi)$.
    There exist integers $a \geq 0 $ and $b$ such that $D_X = a d + b \tau$
    (see for example, Exercise 8 on page 15 in \cite{IR}).
    Then $j \equiv i + s D_X   \equiv i + sad  \;(\text{mod}\;\tau).$
    For $x \in V_i$,
    $\Phi^{*sa} * x \in V_{sad+ i}\otimes\mathbb{C}[q,q^{-1}] = V_{j}\otimes\mathbb{C}[q,q^{-1}]$.
    So quantum multiplication by $\Phi^{*sa}$ gives a linear map
    \[  \Phi^{*sa}* : V_i \otimes\mathbb{C}[q,q^{-1}] \longrightarrow V_{j} \otimes\mathbb{C}[q,q^{-1}].\]

    Because $\Phi$ is invertible, $\Phi^{*(-sa)}=(\Phi^{-1})^{*sa}$ is well-defined and is
homogeneous of degree $-sad$. Since $j-sad\equiv i\pmod{\tau}$,
the quantum multiplication by $\Phi^{*(-sa)}$ defines a map
\[
\Phi^{*(-sa)}*:
V_j\otimes_{\mathbb C}\mathbb C[q,q^{-1}]
\longrightarrow
V_i\otimes_{\mathbb C}\mathbb C[q,q^{-1}],
\]
which is the inverse of $\Phi^{*sa}*$. Hence  $\Phi^{*sa}*$ gives the required
isomorphism.
\end{proof}

This lemma implies that  under the assumption of the existence of an invertible
element with degree equal to $\dim(X)$, all spaces $V_j$ on the right hand side of
inequality \eqref{eqn:DimBoundX-AnyPt}
have the same dimension and they are equal to
\[ \dim (V_0 ) = \sum_{i=0}^{\lfloor \dim(X)/\tau \rfloor} \dim H^{2i\tau}(X).\]
\begin{theorem} \label{thm:DimBoundX}
Let $X$ be a compact symplectic manifold with $H_2(X, \mathbb{Z}) \cong \mathbb{Z}$ and $H^{\rm odd}(X)=0$.
Assume either $\Delta$ or $[pt]$ is invertible in $QH^*(X)_q$. Then
the rank of $\mathfrak{F}_q$ over $\mathbb{C}[q,q^{-1}]$ has the following upper bound:
\begin{equation} \label{eqn:DimBoundX}
 \mathrm{rank} (\mathfrak{F}_q) \leq \frac{\tau}{D_X}  \sum_{i=0}^{\lfloor \dim(X)/\tau \rfloor} \dim H^{2i\tau}(X).
\end{equation}
\end{theorem}
This finishes the proof of the first half of Theorem \ref{thm:  main thm 2}.
We will see in Theorem \ref{thm:  case of Gr(2,n)} that  this bound is sharp for $Gr(2,n)$.

\begin{remark} \label{rem:InvSubSp}
If $[pt]$ is invertible, then $\deg(\Delta/[pt])=0$. Hence quantum multiplication by $\Delta/[pt]$
preserves $V_i \otimes \mathbb{C}[q,q^{-1}]$ for all $i$.
Hence
\[ QH^*(X)_q = \bigoplus_{i=0}^{\tau-1} V_i \otimes \mathbb{C}[q,q^{-1}]\]
 is a decomposition of $QH^*(X)_q$ as a direct sum of invariant submodules of the action
by quantum multiplication of $\Delta/[pt]$.
\end{remark}

\begin{remark}
Note that for all (co)minuscule homogeneous varieties $X$, $[pt]$ is invertible in $QH^*(X)_q$
by Theorem \ref{thm mutiply class of pt}. Moreover $H_2(X, \mathbb{Z}) \cong \mathbb{Z}$ and $H^{\rm odd}(X)=0$.
Hence all results in this section hold for $X$.
\end{remark}

\section{Grassmannians}\label{sec: quantum cohomology}

Let $X=\text{Gr}(k,n)$ be the Grassmannian consisting of all $k$-dimensional subspaces of $\mathbb{C}^n$. This is a
cominuscule homogeneous variety. The small quantum cohomology of $X$ has been well studied in
\cite{ST}, \cite{B}, \cite{BCF}, \cite{Bu}, \cite{P}. In this section, we will give an explicit
formula for the handle element $\Delta$ as a linear combination of Schubert classes
(see Theorem ~\ref{thm:  formula for Delta} and Corollary~\ref{cor:  formula of Delta for Gr(2,n)})
 and give an explicit upper bound for the
dimension of $\mathfrak{F}$ defined by equation \eqref{eqn:F}.
In particular, we will complete the proof for Theorem~\ref{thm:  main thm 2}.

\subsection{Quantum cohomology of Grassmannians}

We first recall some basic facts about quantum cohomology of Grassmannians.
Let $c_i$ be the $i$-th Chern class of the tautological $k$-bundle of $X$ and $s_j$ the $j$-th Chern class of the universal quotient bundle. $s_j$ can be computed from $c_i$ via the following recursion relation:
\begin{equation}\label{eq:inductive definition of s_j}
    s_j=-s_{j-1}\cdot c_1-s_{j-2}\cdot c_2-...-s_0\cdot c_j, \hspace{10pt} {\rm for \,\,\,} j \geq 1,
\end{equation}
\noindent with $c_0=s_0= \mathbf{1}$. Moreover $c_i=0$ if $i >k$ and $s_j=0$ if $j>n-k$. We also set $c_i=s_i=0$ if $i<0$.
For any partition $\lambda=(\lambda_1,...,\lambda_l)$ with integral parts $\lambda_1 \geq...\geq \lambda_l > 0$,
we can define a {\it Schubert class} of (complex) degree $|\lambda|:= \sum_{i=1}^l \lambda_i$ by
\begin{equation}\label{eq:SchbertClass}
\sigma_\lambda := \text{det}(s_{\lambda_i+j-i})_{1\leq i, j\leq l}.
\end{equation}
The number of positive parts of $\lambda$ is called its {\it length} and is denoted by $l(\lambda)$.
Note that adding $0$ to $\lambda$ does not change the value of $\sigma_\lambda$.
We can set $\lambda_i=0$ for all $i>l(\lambda)$.
Define
\begin{equation}\label{eq:definition of P_kn}
    \mathcal{P}_{kn} := \{(\lambda_1,...,\lambda_k) \in \mathbb{Z}^k \vert\;
               n-k\geq\lambda_1\geq\lambda_2\geq...\geq \lambda_k \geq 0\}.
\end{equation}
In other words, $\mathcal{P}_{kn}$ is the set of all partitions whose Young diagrams are contained in the $k\times (n-k)$ rectangle.
Then
$\{ \sigma_{\lambda} \mid \lambda \in {\mathcal P}_{kn} \}$ form a basis of $H^*(X)$ (see, for example, p271 in \cite{Fu}).
Moreover $\sigma_{\lambda} = 0$ if $\lambda \notin \mathcal{P}_{kn}$.
The {\it Poincar\'{e} intersection pairing} on $H^*(X)$ is given by
\begin{equation} \label{eqn:PIPGrassmann}
 \langle \sigma_\mu, \,\, \sigma_{p(\nu)} \rangle = \delta_{\mu \nu}
    \hspace{10pt} {\rm for \,\,\, all \,\,\,} \mu, \nu  \in \mathcal{P}_{kn},
\end{equation}
where $p(\lambda)$ is the {\it complement} of $\lambda=(\lambda_1, \ldots, \lambda_k) \in \mathcal{P}_{kn}$ defined by
\begin{equation}\label{eq: definition of complementary partition}
    p(\lambda):=(n-k-\lambda_k,n-k-\lambda_{k-1},...,n-k-\lambda_1) \in \mathcal{P}_{kn}.
\end{equation}

\noindent
In the quantum cohomology ring $QH^*(X)=H^*(X) \otimes \mathbb{C}[q]$, we have (cf. \cite{ST})
\[ \deg(q)=\tau=n.\]
For any homogeneous elements
$x, y \in H^*(X)$,  $\deg(x * y) = \deg(x)+\deg(y)$. Hence
$x * y$ can not have $q$ factors if $\deg(x)+\deg(y)<n$. Therefore
\begin{equation}\label{eqn: condition for coincidence of usual and quantum products}
     x*y=x\cdot y  \hspace{20pt} {\rm if \,\,\,} \text{deg}(x)+\text{deg}(y)<n.
\end{equation}

Let $\hat{s}_i :=0$ for any $i<0$, $\hat{s}_0 := \mathbf{1}$, and for $j\geq 1$,
    \begin{equation}\label{eq: def of s_j^q}
        \hat{s}_j := -\hat{s}_{j-1}*c_1-\hat{s}_{j-2}*c_2-...-\hat{s}_0*c_j  \,\,\, \in QH^*(X).
    \end{equation}
For any partition $\lambda=(\lambda_1, \ldots, \lambda_l)$, define
\begin{equation} \label{eqn: of hat(sigma)}
    \hat{\sigma}_\lambda := \text{det}^* (\hat{s}_{\lambda_i+j-i})_{1 \leq i,j\leq l} \,\,\, \in QH^*(X),
\end{equation}
where $\text{det}^*$ is the determinant of a matrix where multiplications of entries are given by quantum product.
Note that $(1+c_1+c_2+...)*(1+\hat{s}_1+\hat{s}_2+...)=1$ by definition of $\hat{s}_j$.
So we can use results of Section 14.5 in \cite{Fu} for the quantum cohomology ring. In particular,
since $c_i=0$ for $i>k$, by Formula (3) in Lemma 14.5.1 of \cite{Fu}, we have
\begin{equation} \label{eqn:sigmahat=0}
\hat{\sigma}_\lambda = 0  \hspace{20pt} {\rm if \,\,\,} l(\lambda) > k.
\end{equation}
But $\hat{\sigma}_\lambda$ may not be $0$ if $\lambda_1 > n-k$.
\begin{lemma} \label{lem:QGiabelli}
    $\hat{s}_j = s_j$ for $j < n$ and $\hat{\sigma}_\lambda = \sigma_\lambda$ for $\lambda \in \mathcal{P}_{kn}$.
\end{lemma}

\begin{proof}
Assume $j\leq n-1$. We first show $\hat{s}_j=s_j$  by induction on $j$. By definition, $\hat{s}_i=s_i$ for all $i \leq 0$.
Assume $\hat{s}_i=s_i$ for $i < j$, then  $\hat{s}_i * c_{j-i} = s_i * c_{j-i} = s_i \cdot c_{j-i}$ by
equation \eqref{eqn: condition for coincidence of usual and quantum products}. Comparing equations
\eqref{eq:inductive definition of s_j}  and \eqref{eq: def of s_j^q}, we obtain $\hat{s}_j=s_j$.

If $\lambda = (\lambda_1, \ldots, \lambda_k)  \in \mathcal{P}_{kn}$, then for $1 \leq i, j \leq k$,
\[ \lambda_i + j -i \leq (n-k)+k-1=n-1.\] Hence $\hat{s}_{\lambda_i+j-i} = s_{\lambda_i+j-i}.$
Hence
$\hat{\sigma}_\lambda = \text{det}^* (s_{\lambda_i+j-i})_{1 \leq i,j\leq k} = \sigma_\lambda, $
where the second equality is precisely the quantum Giambelli formula proved in \cite{B}.
The lemma is thus proved.
\end{proof}

By Lemma 14.5.3 in \cite{Fu}, we have
\begin{equation}\label{eq: LR formula for schubert classes}
\hat{\sigma}_\lambda*\hat{\sigma}_\mu=\sum_{\vert\nu\vert=\vert\lambda\vert+\vert\mu\vert} C^\nu_{\lambda\mu}\hat{\sigma}_\nu,
\end{equation}
for any partitions $\lambda$ and $\mu$, where $C^\nu_{\lambda\mu}$ are the Littlewood-Richardson coefficients defined in the
following way: We write $\nu \supset \lambda$ if parts of partitions $\lambda$ and $\nu$ satisfy $\lambda_i \leq \nu_i$ for all $i$.
If $\nu \supset \lambda$,
the {\it skew diagram} $\nu/\lambda$ is obtained by removing the Young diagram of $\lambda$ from the Young diagram of $\nu$.
A {\it semi-standard skew tableau} is a skew diagram with each box labelled by a positive number which is weakly increasing in each row and strictly increasing in each column. The {\it content} of a tableau is a sequence of non-negative
integers $\mu = (\mu_1, \mu_2, \ldots)$ such that $\mu_i$ is the number of $i$'s in the tableau.
We can read the positive integers in the tableau, proceeding in the first row from right to left, then the second row from the right to left, and so on. In this way we get a sequence of integers $\alpha$ called the {\it word} of the tableau. We say that the word is {\it strict} if for any $i\geq 1$ and $j$ not greater than the length of $\alpha$, the number of $i$'s occurring among the first $j$ terms is not less than the number of $(i+1)$'s occurring in these $j$ terms. We call a tableau strict if its associated word is strict.
The {\it Littlewood-Richardson coefficient} $C^\nu_{\lambda\mu}$ is defined to be the
number of strict semi-standard skew tableau with content $\mu$ on the skew diagram $\nu/\lambda$.

\begin{remark}
Note that $\hat{\sigma}_\lambda$ is different from
\[ \omega_\lambda := \text{det}^* (s_{\lambda_i+j-i})_{1 \leq i,j\leq l}, \]
which was used in \cite{BCF} and Sections 1--6 in \cite{Bu}.
In fact, $\omega_\lambda = 0$ if $\lambda_1 > n-k$ and $\omega_\lambda$ may not be $0$ if $l(\lambda)>k$.
We choose $\hat{\sigma}_\lambda$ instead of $\omega_\lambda$ since it is easier to use in computations.
For example, the formula in Lemma~\ref{lemma multiplying Omega of constant partition} has a very simple form
using $\hat{\sigma}_\lambda$. Moreover,
$\hat{\sigma}_\lambda$ agrees with $\sigma_\lambda$  used
in Section 7 in \cite{Bu}. In particular,
by Lemma 4 in \cite{Bu}, we have
\begin{equation} \label{eqn:ReduceIndex}
    \hat{s}_{j+n}=(-1)^{k+1}q\hat{s}_j \hspace{20pt} {\rm for \,\,\,}  j>-k.
\end{equation}
\end{remark}

Note that by equation \eqref{eqn:sigmahat=0}, the summation in equation \eqref{eq: LR formula for schubert classes} can be taken over partitions $\nu$ with $l(\nu)\leq k$. But in general $\nu \notin \mathcal{P}_{kn}$.
Corollary 1 in \cite{Bu} provides a way to replace $\hat{\sigma}_\nu$
by Schubert classes $\sigma_\beta$ with $\beta \in \mathcal{P}_{kn}$.
In fact, for any $I=(I_1,I_2,...,I_l) \in \mathbb{Z}^l$, we can define $\hat{\sigma}_I$  by
 replacing the partition $\lambda$ by $I$ in the definition of
$\hat{\sigma}_\lambda$ given by equation \eqref{eqn: of hat(sigma)}.
If $l \leq k$ and there exists $I_j < j-k$, then $I_j+m-j<0$ and
$\hat{s}_{I_j+m-j} = 0$ for all $1 \leq m \leq l$, which implies $\hat{\sigma}_I=0$.
By the rule for interchanging adjacent rows in the determinant which defines $\hat{\sigma}_I$, we have
\begin{align}\label{eq: rearranging rows}
    \hat{\sigma}_{I,a,a+1,J}&=0,\nonumber\\
    \hat{\sigma}_{I,a,b,J}&=-\hat{\sigma}_{I,b-1,a+1,J}
\end{align}
for any finite sequences of integers $I$ and $J$.
Given any $\nu$ with $l(\nu) \leq k$, we can use equation \eqref{eqn:ReduceIndex} to reduce
indices for $\hat{s}_*$ appeared in the definition of $\hat{\sigma}_{\nu}$ and obtain
\begin{proposition}[Corollary 1 in \cite{Bu}]
\label{prop: rearranging partition}
    Let $\nu$ be a partition with $l(\nu)\leq k$. Choose $I_j\in \mathbb{Z}$ such that
    $I_j\equiv \nu_j \text{ }(mod\text{ }n)$ and
    $j-k\leq I_j < j-k+n$ for each $1 \leq j \leq k$. Then we have
    \[
    \hat{\sigma}_\nu=(-1)^{r(k+1)}q^r\hat{\sigma}_I,
    \]
where $I=(I_1,I_2,...,I_k)$ and $r=\left( \vert\nu\vert-\sum_{j=1}^k I_j \right)/n \,\, \in \mathbb{Z}$.
\end{proposition}
\noindent
By equation \eqref{eq: rearranging rows} and Lemma \ref{lem:QGiabelli}, for $I$ in the above proposition,
$\hat{\sigma}_I =  \epsilon \sigma_\beta$ for some
$\beta \in \mathcal{P}_{kn}$ and $\epsilon=0, \pm 1$.

The following formula will be useful in later calculations.

\begin{lemma}\label{lemma multiplying Omega of constant partition}
    If $\lambda=(a^k)$ is a partition with $k$ parts all equal to $a$, then
    \[\hat{\sigma}_\lambda*\hat{\sigma}_\mu=\hat{\sigma}_{(\mu_1+a,\mu_2+a,...,\mu_k+a)},\]

    \noindent for any partition $\mu=(\mu_1, \ldots, \mu_k)$ with $l(\mu)\leq k$.
\end{lemma}

\begin{proof}
     By equation \eqref{eq: LR formula for schubert classes}, we need to compute $C_{\lambda \mu}^\nu$ for
      partitions $\nu \supset \lambda$ with $l(\nu)\leq k$ and $|\nu|=|\lambda|+|\mu|$. Let $\mathcal{T}$ be a strict semi-standard skew tableau of shape $\nu/\lambda$ with content $\mu$.

      Step 1. Prove that no box in the first row can be labelled by a number $>1$.

            If such box exists, then
    the rightmost box of the first row must be labelled by an integer $m>1$ since labels must be weakly increasing along each row. So the first term of the word for $\mathcal{T}$ contains one  label $m$ but no label $m-1$. This contradicts to the strictness of $\mathcal{T}$. Note that this argument works for all strict semi-standard skew tableau.

       Step 2. Prove that all boxes in $\mathcal{T}$ labelled by $1$ must be in the first row.

       Suppose there is a
       box labelled by 1 which is not in the first row. Since $\lambda = (a^k)$, the skew diagram
      $\nu/\lambda$ must be the Young diagram of some partition. Hence there must be a box located right above this one which is labelled by some number $\geq 1$. So the labels along this column can not be strictly increasing. This contradicts the assumption that $\mathcal{T}$ is semi-standard.

      Combining results of steps 1 and 2, we proved that the first row of $\mathcal{T}$ must contain exactly $\mu_1$ boxes which are all labelled by $1$.
      Repeating the above  arguments for each row consecutively, we see that for every $1 \leq i \leq k$, the $i$-th row of $\mathcal{T}$ must contain exactly $\mu_i$ boxes which are all labelled by $i$. Hence $\nu=(a+\mu_1, \ldots, a+\mu_k)$
      and $C_{\lambda \mu}^\nu=1$. The lemma is thus proved.
\end{proof}

Let
\begin{equation} \label{eqn:D1D2}
D_1:=\text{gcd}(k,n), \hspace{20pt} D_2:=\text{gcd}(k^2,n).
\end{equation}
Since $[pt]=\sigma_{((n-k)^k)}$, by Lemma \ref{lemma multiplying Omega of constant partition} and
Proposition \ref{prop: rearranging partition}, we have
\[ [pt]^{*\frac{n}{D_1}}
=\hat{\sigma}_{(((n-k)n/D_1)^k)}
=(-1)^{(k+1)k\frac{(n-k)}{D_1}}q^{\frac{k(n-k)}{D_1}} \,\, \sigma_{(0^k)}.\]
Note that $\sigma_{(0^k)} = \mathbf{1}$ is the identity of the cohomology ring. Hence we have
\begin{equation} \label{eqn: Omega_per is n-periodic}
    [pt]^{*\frac{n}{D_1}}=q^{\frac{k(n-k)}{D_1}} \mathbf{1}.
\end{equation}
The number $D_2$ in equation \eqref{eqn:D1D2} is a special case of $D_X$ defined by equation \eqref{eqn:DX}.
In fact, for $X=\text{Gr}(k,n)$, $\text{dim}(X)=k(n-k)$, $\text{deg}(q)=\tau=n$. So
$D_X =\text{gcd}(n, k(n-k)) =D_2$.

\begin{remark}
Lemma \ref{lemma multiplying Omega of constant partition} is similar (but has simpler form) to the result in
Proposition 6.3 in \cite{P}. Equation \eqref{eqn: Omega_per is n-periodic} also follows from Proposition 6.3 in \cite{P}.
\end{remark}

\subsection{Computing $\Delta$ for Grassmannians}\label{section computing Delta}

 In this subsection,  we give a formula expressing the handle element $\Delta$ as a linear combination of Schubert classes.

Let $R:= \lfloor {\frac{k(n-k)}{n}}\rfloor$.
For $0 \leq r \leq R$, define
\[ \mathcal{P}_{kn}^r := \{ \nu \in \mathcal{P}_{kn} \mid |\nu|=k(n-k)-rn\}. \]
Then $\mathcal{P}_{kn}^0 = \{ [pt] \}$ where $[pt]=\sigma_{((n-k)^k)}$ is the point class.
For any integer $1 \leq r \leq k$, let
\[ Z_r := \{(i_1, \ldots, i_r) \in \mathbb{Z}^r \mid 1 \leq i_1 < i_2 < \cdots < i_r \leq k\}.\]
For any $I=(i_1, \ldots, i_r)$, define $|I|:= i_1 + \cdots + i_r$.
For $1 \leq r \leq R$, $\nu \in \mathcal{P}_{kn}^r$, $I=(i_1, \ldots, i_r) \in Z_r$,
 we define a partition $\varphi(\nu, I) = (\varphi_1, \ldots, \varphi_k)$ by
\begin{equation} \label{eqn:NuItoMu}
\varphi_j = \left\{ \begin{array}{ll}
                \nu_{i_j} - i_j + j +n, & {\rm if \,\,\,}  1 \leq j \leq r, \\
                \nu_{j-r+l-1} + r-l+1, & {\rm if \,\,\,} i_{l-1}-l+2 \leq j-r \leq i_l -l, \,\,\, 1\leq l \leq r, \\
                \nu_j,  &{\rm if \,\,\,} i_r+1 \leq j \leq k,
                \end{array} \right.
\end{equation}
where $i_0:=0$. Then we have
\begin{theorem}\label{thm: formula for Delta}
    For $X=\text{Gr}(k,n)$, the handle element $\Delta$ is given by
\begin{equation}\label{eq: formula for Delta}
    \Delta= \chi(X)  [pt] + \sum_{r=1}^R q^r\sum_{\nu\in \mathcal{P}_{kn}^r}
           \sum_{I \in Z_r}(-1)^{\frac{r(2k-r+1)}{2}+|I|}  \sum_{\lambda\in\mathcal{P}_{kn}}
               C^{\varphi(\nu, I)}_{\lambda \, \, p(\lambda)} \,\, \sigma_\nu,
\end{equation}

    \noindent where $\chi(X)=\dim H^*(X)$ is the Euler  characteristic number of $X$.
\end{theorem}

\begin{proof}

By equation \eqref{eqn: form of Delta}, we have

\begin{equation}\label{eq:original expression for delta}
    \Delta=\sum_{\lambda\in \mathcal{P}_{kn}} \sigma_\lambda*\sigma_{p(\lambda)}
          =\sum_{\lambda\in \mathcal{P}_{kn}}\sum_{\mu \in A}  C^\mu_{\lambda p(\lambda)}\hat{\sigma}_\mu,
\end{equation}

\noindent where $A$ is the set of all partitions $\mu$ with
 $\vert\mu\vert=\vert\lambda\vert+\vert p(\lambda)\vert=k(n-k)$
 and $l(\mu) \leq k$.

When expressing $\Delta$ as a linear combination of Schubert classes $\sigma_{\nu}$ with $\nu \in \mathcal{P}_{kn}$,
the coefficient of $q^0$ is given by
$\sum_{\lambda\in \mathcal{P}_{kn}} \sigma_\lambda \cdot \sigma_{p(\lambda)} = \chi(X) [pt]$.
In the rest part of the proof, we only need to consider coefficient of $q^r$ for $r \geq 1$.

If $C^\mu_{\lambda p(\lambda)} \neq 0$, then $\mu \supset \lambda$ and there exists a strict semi-standard skew tableau $\mathcal{T}$ of shape $\mu/\lambda$ with content $p(\lambda)$. By the result in Step 1 in the proof of
Lemma \ref{lemma multiplying Omega of constant partition},
all boxes in the first row of $\mathcal{T}$ must be labeled by 1. This implies that
the number of boxes in the first row of $\mathcal{T}$ is less than or equal to $p(\lambda)_1$ which is the first part of
the $p(\lambda)$.
Hence $ \mu_1 \leq \lambda_1+p(\lambda)_1 \leq 2(n-k) < 2n-k$. So
\begin{equation}\label{eq:DeltaB}
    \Delta= \chi(X) [pt] + \sum_{\lambda\in \mathcal{P}_{kn}}\sum_{\mu \in B}  C^\mu_{\lambda p(\lambda)}\hat{\sigma}_\mu,
\end{equation}
where $B$ is the set of all partitions $\mu=(\mu_1, \ldots, \mu_k)$ satisfying properties:
\[ \mu_1 \leq 2n-k, \hspace{10pt}
 \vert\mu\vert=k(n-k), \hspace{10pt}
 \mu \notin \mathcal{P}_{kn}, \hspace{10pt}
 \hat{\sigma}_\mu \neq 0.\]

For any partition $\mu=(\mu_1, \ldots, \mu_k) \in B$, we can find integers $J=(J_1, \ldots, J_k)$ such that
\[ J_i \leq \mu_i, \,\,\, J_i \equiv \mu_i ({\rm mod \,\,\,} n), \,\,\, i-k \leq J_i < i-k+n \]
 for all $1 \leq i \leq k$. Let $r= (|\mu|-|J|)/n \leq R$. Then
by Proposition \ref{prop: rearranging partition}, we have
\begin{equation} \label{eqn:mutoJ}
 \hat{\sigma}_\mu = (-1)^{r(k+1)} q^r \hat{\sigma}_J.
\end{equation}
If $r=0$, then $\mu = J \in \mathcal{P}_{kn}$ and $\hat{\sigma}_\mu = \sigma_\mu$ is a Schubert class.
So we can assume $r>0$.

Since $\mu_i \leq \mu_1 \leq 2n-k$, we have $0 \leq \mu_i - J_i \leq 2n -i < 2n$
 for all $1\leq i \leq k$. Since $J_i \equiv \mu_i ({\rm mod \,\,\,} n)$,
  $\mu_i - J_i$ must be either $0$ or $n$ for all $i$.
 Moreover, $\mu_i - J_i = n$ if and only if $\mu_i \geq i-k+n$.
 Since $\mu$ is a partition,  $\mu_i - i$ is decreasing as $i$ increases.
 Therefore if there is an $i$ such that $\mu_i \geq i-k+n$, then
 for all $1 \leq j \leq i$,
  $\mu_j \geq j-k+n$, and consequently $\mu_j - J_j = n$.
 Since $|\mu|-|J|=rn$, we must have
 \[\mu_i - J_i = \left\{ \begin{array}{ll} n, & {\rm \,\,\, if \,\,\,} 1 \leq i \leq r, \\
                                            0, & {\rm \,\,\, if \,\,\,} r+1 \leq i \leq k.
                        \end{array} \right. \]
  We now use  equation \eqref{eq:  rearranging rows} to rearrange the indices of $\hat{\sigma}_J$
  to get a partition $\nu \in \mathcal{P}_{kn}$ such that  $\hat{\sigma}_J = \epsilon \sigma_\nu$ with
  $\epsilon=0, \pm 1$.
  Since $\{ J_1, \ldots, J_r\}$ and $\{J_{r+1}, \ldots, J_k\}$ are both weakly decreasing, we can keep the relative order
  of the  corresponding indices  within each of these two sets during this process.
  There are integers $I=(i_1, \ldots, i_r) \in Z_r$ such that
   $\nu_{i_j}$ is obtained from $J_j$ in this process for $j=1, \ldots, r$.
   In this process,  the number of times we need to use equation~\eqref{eq:  rearranging rows}
   is $\sum_{j=1}^r (i_j-j)= |I|- r(r+1)/2$. Hence we have
   $\hat{\sigma}_J = (-1)^{|I|-r(r+1)/2} \sigma_\nu$. Combining with equation \eqref{eqn:mutoJ},   we have
\begin{equation}\label{eq: relation between mu and nu}
    \hat{\sigma}_\mu=(-1)^{|I| +\frac{r(2k-r+1)}{2}} q^r \sigma_\nu.
\end{equation}
In the above process, both $\nu$ and $I$ are uniquely determined by $\mu$.
Note that $r=(|\mu|-|\nu|)/n$ is also uniquely determined by $\mu$.
Hence we obtain a map
\begin{equation} \label{eqn:FAB}
    \begin{array}{rccl}
        F: & B & \longrightarrow & \bigcup_{r=1}^{R} \mathcal{P}_{kn}^r \times Z_r \\
           & \mu & \longmapsto & (\nu, I).
    \end{array}
\end{equation}
On the other hand, given any $1 \leq r \leq R$, $\nu \in \mathcal{P}_{kn}^r$, $I=(i_1, \ldots, i_r) \in Z_r$,
we can obtain a partition $\varphi(\nu, I)$
by first pushing the $i_1, \ldots, i_r$-th parts of $\nu$ to the beginning of the partition
using the rules for switching indices as in equation \eqref{eq:  rearranging rows},
  and then adding $n$ to the first $r$ parts. The formula for $\varphi(\nu, I)$ is given
  by equation \eqref{eqn:NuItoMu}.
It is straightforward to check $\varphi(\nu, I) \in B$ and
$F(\varphi(\nu, I)) = (\nu, I)$.
Hence $F$ is a bijection with inverse given by $\varphi$.
The theorem then follows from equations \eqref{eq:DeltaB}
and \eqref{eq: relation between mu and nu}.
\end{proof}

We now give a more explicit formula for $\Delta$ when $X=\text{Gr}(2,n)$.
We first compute the Littlewood-Richardson coefficients for this case.

\begin{lemma} \label{lem:LRk=2}
For any partitions $\lambda=(\lambda_1, \lambda_2)$, $\mu=(\mu_1, \mu_2)$, and $\nu=(\nu_1, \nu_2)$,
if
\begin{equation} \label{eqn:CondC=1}
\mu \supset \lambda, \hspace{10pt}
|\mu| = |\lambda| + |\nu|, \hspace{10pt}
\nu_1 \geq \mu_1 - \lambda_1, \hspace{10pt}
\mu_2 - \lambda_1 \leq \nu_2 \leq \mu_1 - \lambda_1,
\end{equation}
then $C_{\lambda \nu}^\mu=1$ .
Otherwise, $C_{\lambda \nu}^\mu=0$.
\end{lemma}
\begin{proof}
Suppose there exists  a strict semi-standard skew tableau $\mathcal{T}$ of shape $\mu/\lambda$ and with content $\nu$.
The existence of $\mathcal{T}$ requires $\mu \supset \lambda$ and $|\mu| = |\lambda| + |\nu|$.
Step 1 in the proof of Lemma \ref{lemma multiplying Omega of constant partition} shows that all boxes in the
first row must be labelled by 1. This implies that $\nu_1 \geq \mu_1-\lambda_1$.
Since $\mathcal{T}$ is semi-standard, all boxes labelled by 2 must be on the second row and must be on the right of
boxes labelled by 1 on the same row. Hence $\mathcal{T}$ is unique if it exists. This implies that
$C_{\lambda \nu}^\mu=1$ if $\mathcal{T}$ exists, and $C_{\lambda \nu}^\mu=0$ otherwise.
Since labels along each column are strictly increasing, each box labelled by 1 must be on the top of its column in
$\mathcal{T}$. Hence the number of boxes labelled by 2 is bigger than or equal to $\mu_2-\lambda_1$. By strictness of
$\mathcal{T}$, number of all boxes labelled by 2 must be less than or equal to the number of boxes in the first row
which are labelled by 1. Hence we have $\mu_2 - \lambda_1 \leq \nu_2 \leq \mu_1 - \lambda_1$.
This shows that conditions \eqref{eqn:CondC=1} must be satisfied if $\mathcal{T}$ exists. On the other hand, if these
conditions are satisfied, we can construct $\mathcal{T}$ in a unique way. The lemma is thus proved.
\end{proof}

The following result is a consequence of Theorem \ref{thm:  formula for Delta}.
\begin{corollary} \label{cor: formula of Delta for Gr(2,n)}
    For $X=\text{Gr}(2,n)$ with $n\geq4$, we have

    \[
    \Delta=\frac{n(n-1)}{2}\sigma_{(n-2,n-2)}+\sum^{\lfloor\frac{n-2}{2}\rfloor}_{s=1}\frac{n(n-2s-1)}{2} q\sigma_{(n-3-s,s-1)}.
    \]
\end{corollary}

\begin{proof}
In this case $\chi(X)$ is the number of partitions in $\mathcal{P}_{2n}$, which is equal to $\frac{n(n-1)}{2}$.
Moreover $R=\lfloor {\frac{2(n-2)}{n}}\rfloor=1$ for $n\geq 4$. By equation \eqref{eq:  formula for Delta}, we only need to
compute coefficient of $q$ in $\Delta$, i.e. the case $r=1$.
For $\nu\in \mathcal{P}_{2n}^1$, we have $\vert\nu\vert=n-4$.
So $\nu= \nu(s) = (n-3-s, s-1)$ for $1\leq s\leq \lfloor\frac{n-2}{2}\rfloor$.
$I$ is now a single integer $i$ which is either 1 or 2. Moreover
\[ \varphi(\nu(s), 1)=(2n-3-s, s-1), \hspace{20pt} \varphi(\nu(s), 2)=(n+s-2,n-s-2).\]
Let $N_i(s)$ be the number of $\lambda \in \mathcal{P}_{2n}$ such that $C_{\lambda \, p(\lambda)}^{\varphi(\nu(s), i)} = 1$.
By Theorem \ref{thm:  formula for Delta},
\begin{equation} \label{eqn:DeltaN1N2}
\Delta = \frac{n(n-1)}{2} \sigma_{(n-2, n-2)}
         + \sum_{s=1}^{\lfloor\frac{n-2}{2}\rfloor} (N_2(s) - N_1(s)) q  \sigma_{(n-3-s, s-1)}.
\end{equation}

We first compute $N_1(s)$. By Lemma \ref{lem:LRk=2}, $C_{\lambda \, p(\lambda)}^{\varphi(\nu(s), 1)} = 1$ if and only if
$\lambda_1 \leq n-2$, $\lambda_2 \leq s-1$, and $\lambda_1  \geq  \lambda_2 + n-1-s$. Hence
\[ N_1(s) \,\,=\,\, \sum_{\lambda_2 = 0}^{s-1} \,\,\, \sum_{\lambda_1 = \lambda_2 + n-1-s}^{n-2} 1
    \,\,=\,\, \frac{s(s+1)}{2}.\]

We now compute $N_2(s)$. By Lemma \ref{lem:LRk=2}, $C_{\lambda \, p(\lambda)}^{\varphi(\nu(s), 2)} = 1$ if and only if
$\lambda_1 \leq n-2$, $\lambda_2 \leq n-s-2$, and  $\lambda_1  \geq  \lambda_2 + s$. Hence
\[ N_2(s) \,\,=\,\, \sum_{\lambda_2 = 0}^{n-s-2} \,\,\, \sum_{\lambda_1 = \lambda_2 + s}^{n-2} 1
    \,\,=\,\, \frac{(n-s)(n-s-1)}{2}.\]
Therefore $N_2(s)-N_1(s) = \frac{n(n-2s-1)}{2}$ and the corollary follows from equation \eqref{eqn:DeltaN1N2}.
\end{proof}

\subsection{Estimate the dimension of $\mathfrak{F}$ for Grassmannians}\label{sec: characterization of Delta^i}

In this subsection, we give an estimate for the dimension of $\mathfrak{F}$ which is defined by equation \eqref{eqn:F}.
Recall that for $X=\text{Gr}(k,n)$, $\text{dim}(X)=k(n-k)$ and $\text{deg}(q)=\tau=n$.
After setting $q=1$ in Theorem \ref{thm: subspace spanned by Delta^i} and applying it to $X=Gr(k,n)$, we have
\begin{equation} \label{eqn:SpaceContainFGr}
 \mathfrak{F} \subset \bigoplus_{\substack{j \equiv 0\;(mod\;D_2) \\ 0 \leq j \leq n-1}}   V_j,
\end{equation}
where $D_2 := {\rm gcd}(n, k^2)$ and
$V_j := {\rm Span}_{\mathbb{C}} \{ \sigma_\lambda
     \mid \lambda \in \mathcal{P}_{kn}, |\lambda| \equiv j ({\rm mod} \,\,\, n)\}$.

Let $p(i \mid m, l)$ be the number of partitions $\lambda$  such that $|\lambda|=i$ and the Young diagram of $\lambda$
is contained in an $l \times m$ rectangle.
Then \[ \dim(H^{2i}(X))=p(i \mid n-k, k). \]
Hence in this case, Theorem \ref{thm:DimBoundX} can be written as
\begin{theorem} \label{thm: estimation of dimension of Span(Delta^i)}
    For $X=Gr(k,n)$,
    \begin{equation} \label{eqn:UBDimF}
     \dim(\mathfrak{F})  \leq  \frac{n}{D_2}\cdot\sum_{i=0}^{\lfloor\frac{k(n-k)}{n}\rfloor} p(in \mid n-k,k).
    \end{equation}
\end{theorem}

The generating function for numbers $p(i \mid m, l)$ is the Gaussian binomial coefficients
$\left[ \begin{array}{c} m+l \\ m \end{array} \right]_q$ whose value at $q=1$ is $\binom{m+l}{m}$,
which is also the number of partitions whose Young diagram is contained in the $m \times l$ rectangle
(cf. Proposition 1.1 in \cite{Ai}). A standard property of Gaussian binomial coefficients (i.e. equation (7) on page 39 in
\cite{Ai}) implies the following recursion formula for $p(i \mid m, l)$:
\begin{equation} \label{eq: inductively compute number of partitions}
    p(i \mid m,l)= p(i \mid m-1, l) + p(i-m \mid m, l-1)
\end{equation}
\noindent with $p(0 \mid m,l)=1$ for all $m, l \geq0$, $p(i \mid m,l)=0$ for $i>0$ and at least one of $m,l$ equal to 0.

To get an asymptotic estimate for the right hand side of the inequality \eqref{eqn:UBDimF},
we first prove the following lemma.

\begin{lemma} \label{lem:AsymptBound}
For any fixed $k \geq 0$,
\begin{equation} \label{eqn:AsymptBound}
 \lim_{n \rightarrow \infty} \frac{\sum_{i=0}^{\lfloor\frac{k(n-k)}{n}\rfloor} p(in \mid n-k,k)}{n^{(k-1)}} = \frac{1}{k!}.
\end{equation}
\end{lemma}

\begin{proof}
The following asymptotic formula was proved by Stanley and Zanello (see Theorem 2.4 in \cite{SZ}): For fixed $i, r, k$,
as $n \rightarrow \infty$,
\[
p(in+r \mid k, n-k) = \frac{1}{k!(k-1)!} A_{k-1, i} (n-k)^{k-1} + O((n-k)^{k-2}),
\]
where $A_{k-1, i} := \sum_{j=0}^i (-1)^j \binom{k}{j} (i-j)^{k-1}$ is the Euler number.
Hence, for $n > k^2$, the left hand side of equation \eqref{eqn:AsymptBound} is equal to
\[  \sum_{i=0}^{k-1} \lim_{n \rightarrow \infty} \frac{ p(in \mid n-k,k)}{n^{(k-1)}}
 =  \sum_{i=0}^{k-1} \frac{1}{k!(k-1)!} A_{k-1, i}
 = \frac{1}{k!}, \]
 where the last equality follows from a standard property of Euler numbers: $\sum_{i=0}^{k-1} A_{k-1, i} = (k-1)!$
 (see Exercise 1.49 on page 30 in \cite{Ai}).
 The lemma is thus proved.
\end{proof}

Since $\dim H^*(X) =  \binom{n}{k}$, which is equal to the number of partitions with Young diagram contained in the
$k \times (n-k)$ rectangle, we have
\[ \lim_{n \rightarrow \infty} \frac{\dim H^*(X)}{n^k} = \frac{1}{k!}.\]
An immediate consequence of Theorem \ref{thm: estimation of dimension of Span(Delta^i)} and Lemma \ref{lem:AsymptBound}
is the following

\begin{corollary} \label{cor:SymptBound}
For fixed $k \geq 1$,
\[ \overline{\lim_{n \rightarrow \infty}} \,\,\, D_2 \,\, \frac{\dim(\mathfrak{F})}{\dim H^*(X)} \leq 1. \]
\end{corollary}
This result tells us that for n large enough, $\frac{\dim(\mathfrak{F})}{\dim H^*(X)}$ is approximately bounded from above by
$\frac{1}{D_2} = \frac{1}{\rm gcd(n, k^2)}$.

\begin{example} \label{ex:comparison}
   For $X=Gr(k,n)$,  we denote the right hand side of inequality~\eqref{eqn:UBDimF} by $Est(X)$.
   The following table gives a comparison between $\dim(H^*(X))$ and $Est(X)$ for some Grassmannians.

    \begin{center}
        \begin{tabular}{ | l | c | r | }
  \hline
  $X$ & $dim(H^*(X))$ & $Est(X)$ \\
  \hline
  $Gr(2,4)$ & 6 & 2 \\
  $Gr(2,5)$ & 10 & 10 \\
  $Gr(2,6)$ & 15 &  9   \\
  $Gr(2,7)$ & 21 & 21 \\
  $Gr(2,8)$ & 28 & 8 \\
  $Gr(3,6)$ & 20 & 8 \\
  $Gr(3,7)$ & 35 & 35 \\
  $Gr(3,8)$ & 56 & 56 \\
  $Gr(3,9)$ & 84 & 9 \\
  $Gr(4,8)$ & 70 & 10 \\
  \hline
\end{tabular}
    \end{center}

From this table, we can see that there are many examples where $\dim(\mathfrak{F})$ is much
smaller than $\dim(H^*(X))$.

\end{example}

\subsection{Case of $\text{Gr}(2,n)$}\label{sec: cases of Gr(2,n)}

In this subsection, we will prove that for $X=Gr(2,n)$,
the inclusion relation \eqref{eqn:SpaceContainFGr} (or the inclusion in Theorem \ref{thm: subspace spanned by Delta^i})  is an equality. This implies that the inequality \eqref{eqn:UBDimF} is also an equality and we obtain a precise
formula for $\dim(\mathfrak{F})$ for $Gr(2,n)$. Since $Gr(2, 3)$ is just the projective plane which has been discussed in Section \ref{sec:Pn}, we will assume $n \geq 4$ and set $m=\lfloor \frac{n}{2}\rfloor$.

Recall $V_0 := \bigoplus_{i=0}^{\lfloor \dim(X)/\tau \rfloor} H^{2i\tau}(X)$ is the subspace which
was used to estimate $\dim(\mathfrak{F})$ in Theorem \ref{thm:DimBoundX}.
For $X=Gr(2,n)$, $\dim(X)=2n-4$ and $\tau=n$. So
$V_0 = {\rm Span}\{\Theta_i \mid i=1, \ldots, m\}$ where
\[ \Theta_1 := \sigma_{(0,0)}=\mathbf{1}, \hspace{20pt}
   \Theta_j := \sigma_{(n-j,j)} {\rm \,\,\, for \,\,\,} 2\leq j \leq m. \]
By Remark \ref{rem:InvSubSp}, $V_0 \otimes \mathbb{C}[q, q^{-1}]$ is an invariant submodule of the quantum multiplication
by $\Delta/[pt]$. Let $A_0 = (A_{ij})_{1 \leq i, j \leq m}$ be the matrix of this action on $V_0 \otimes \mathbb{C}[q, q^{-1}]$ with respect to the basis
\[ \Theta^q := \{ \Theta_1, q^{-1} \Theta_2, \ldots, q^{-1} \Theta_m\}.\]
 Note that $\Theta^q$
agrees with the basis defined in equation \eqref{eqn:f(w)} since $y(w)=1$ for all $w$ when $X$ is a Grassmannian.
By Remark \ref{rem:InvSubSp} and Proposition \ref{prop: main result that A is symmetric}, $A_0$ is a positive definite
real symmetric matrix.

By Corollary \ref{cor: formula of Delta for Gr(2,n)},
\begin{equation} \label{eqn:DeltaGr2n}
        \Delta  = \sum_{i=1}^m b_i q^{1-\delta_{i1}} \Lambda_i,
\end{equation}
where
\[ \Lambda_1 := [pt]=\sigma_{(n-2,n-2)}, \hspace{20pt} \Lambda_j := \sigma_{(n-j-2, j-2)} \]
for $2\leq j \leq m$, and
\begin{equation} \label{eqn:bi}
 b_i := n(n+1)/2 - in
\end{equation}
 for $1 \leq i \leq m$.
To compute $A_0$, we need the following formula

\begin{lemma} \label{lem:ThetaLambda}
For $2 \leq i, j \leq m$,
\[ \Theta_i * \Lambda_j = \sum_{k=|i-j|+1}^{\min \{i+j-1, \,\, n+1-i-j\}} q^{1-\delta_{k1}} \Lambda_k.\]
\end{lemma}
\begin{proof}
By equation \eqref{eq: LR formula for schubert classes},
\[
\Theta_i * \Lambda_j = \sum_{\substack{\lambda_1 \geq \lambda_2 \geq 0 \\ \lambda_1 + \lambda_2 = 2n-4 } }
                    C_{(n-i, i), (n-j-2, j-2)}^{(\lambda_1, \lambda_2)}
                            \,\, \, \hat{\sigma}_{(\lambda_1, \lambda_2)},
\]
where $C_{(n-i, i), (n-j-2, j-2)}^{(\lambda_1, \lambda_2)}$ is the Littlewood-Richardson coefficient.
By Lemma \ref{lem:LRk=2}, this coefficient is either 0 or 1, and it is equal to 1 if and only if $(\lambda_1, \lambda_2)$
satisfies the following conditions: $\lambda_1 + \lambda_2 = 2n-4$ and
\[
\lambda_2 \geq i, \hspace{10pt}
n-j-2 \geq \lambda_1-n+i, \hspace{10pt}
\lambda_2-n+i \leq j-2 \leq \lambda_1 -n + i.
\]
These condition is equivalent to $ i+j-2 \leq \lambda_2 \leq n-2 - |i-j|$. Set $\lambda_2=k$. We have
\begin{eqnarray*}
\Theta_i * \Lambda_j
&=& \sum_{k=i+j-2}^{n-2-|i-j|} \,\, \, \hat{\sigma}_{(2n-4-k, k)}.
\end{eqnarray*}
By Proposition \ref{prop:  rearranging partition}, for $ i+j-2 \leq k \leq m-2$,  we have
\[ \hat{\sigma}_{(2n-4-k, k)} = -q \hat{\sigma}_{(n-4-k, k)} = -q \Lambda_{k+2}, \]
where the last equality also used Lemma \ref{lem:QGiabelli}.
For  $ m-1 \leq k \leq n-3$, we first use Proposition \ref{prop:  rearranging partition}, then use equation
\eqref{eq:  rearranging rows} to obtain
\[ \hat{\sigma}_{(2n-4-k, k)} = -q \hat{\sigma}_{(n-4-k, k)} = q \hat{\sigma}_{(k-1, n-3-k)} = q \Lambda_{n-1-k}.\]
In this formula we should set $\Lambda_{m+1}:=0$ if $n=2m+1$. Hence we have
\begin{eqnarray*}
\Theta_i * \Lambda_j
&=& \delta_{ij}  \Lambda_{1}  - q \sum_{k=i+j-2}^{m-2} \Lambda_{k+2}
    + q \sum_{k=\max\{m-1,i+j-2\}}^{\min\{n-3, n-2-|i-j|\}} \Lambda_{n-1-k}.
\end{eqnarray*}
This implies the desired formula after obvious cancellation.
\end{proof}

Now we can compute the matrix $A_0=(A_{ij})$ of quantum multiplication by $\Delta/[pt]$ with respect to the basis
$\Theta^q$.

\begin{proposition}
\label{prop: explicit form of A_0 for Gr(2,2m)}

    $A_0=(A_{ij})_{1 \leq i, j\leq m}$ is a symmetric matrix with
    \[ A_{ij} = (2i-1) b_j \]
    for all $1 \leq i \leq j \leq m$, where $b_j$ is given by equation \eqref{eqn:bi}.
\end{proposition}

\begin{proof}

Note that $[pt]=\Lambda_1=\sigma_{(n-2,n-2)}$.
By Lemma \ref{lemma multiplying Omega of constant partition} and Proposition \ref{prop:  rearranging partition}, we have
\begin{equation} \label{eqn:[pt]*k=2}
 [pt]* \sigma_{(\lambda_1, \lambda_2)} = \hat{\sigma}_{(\lambda_1+n-2, \lambda_2+n-2)}
 = q^2 \sigma_{(\lambda_1-2, \lambda_2-2)}
\end{equation}
if $n-2 \geq \lambda_1 \geq \lambda_2 \geq 2$.
Hence
\begin{equation} \label{eqn:ptTheta}
[pt]* \Theta_i = q^{2-2 \delta_{i1}} \Lambda_i  \hspace{20pt} {\rm for \,\,\,} 1 \leq i \leq m.
\end{equation}
So the matrix $A_0=(A_{ij})$ is given by

\begin{equation*}
        \Delta* q^{ \delta_{i1}-1} \Theta_i = \sum_{j=1}^m A_{ij} q^{1- \delta_{j1}} \Lambda_j
\end{equation*}
for all $1 \leq i \leq m$.
By equation \eqref{eqn:DeltaGr2n}, we have $A_{1j} = b_j$ for all $1 \leq j \leq m$, where $b_j$ is defined by
equation \eqref{eqn:bi}.

Assume $2 \leq i \leq m$. By equations \eqref{eqn:DeltaGr2n}, \eqref{eqn:ptTheta}, and Lemma \ref{lem:ThetaLambda}, we have
    \begin{align*}
       \Delta * q^{-1} \Theta_i
       &= b_1 q \Lambda_i + \sum_{j=2}^m b_j \sum_{k=|i-j|+1}^{\min \{i+j-1, \,\, n+1-i-j\}} q^{1-\delta_{k1}} \Lambda_k.
    \end{align*}
Taking coefficients of $\Lambda_1$ on both sides, we obtain $A_{i1}=b_i$ for all $2 \leq i \leq m$.
Taking coefficients of $q\Lambda_k$ on both sides for $2 \leq k \leq m$, we have
\[ A_{ik} = b_1 \delta_{ik} + \sum_j b_j, \] where the summation is over all $j$ satisfying conditions
\[
j \geq 2, \hspace{20pt}
|i-j|+1 \leq k \leq \min \{i+j-1, \,\, n+1-i-j\},
\]
which are equivalent to
\[ j \geq 2, \hspace{20pt}
1+|i-k| \leq j \leq \min\{n+1-i-k, \,\, i+k-1\}.\]
Hence for $1 \leq i, j \leq m$,  we have
    \begin{equation} \label{eqn:Aij}
    A_{ij}=\sum_{k=|j-i|+1}^{\min\{i+j-1, \,\, n-i-j+1\}} b_k.
    \end{equation}
This implies that $A_0$ is a symmetric matrix.

Since $A_{1j} = b_j$ for all $1 \leq j \leq m$, to prove the proposition, we only need to show
$A_{ij}-A_{(i-1)j} = 2 b_j$ for all $2 \leq i \leq j \leq m$. By equation \eqref{eqn:Aij},
if $2(i+j) \leq n+2$,
\[ A_{ij}-A_{(i-1)j} = b_{j-i+1} + b_{i+j-1} = 2 b_j.\]
If $2(i+j) \geq n+4$,
\[ A_{ij}-A_{(i-1)j} = b_{j-i+1} - b_{n-i-j+2} = 2 b_j.\]
If $n$ is odd and $2(i+j) = n+3$,
\[ A_{ij}-A_{(i-1)j} = b_{j-i+1} = 2 b_j.\]
The proposition is thus proved.
\end{proof}

\begin{lemma} \label{lem:A0Mult=1}
(i)  Multiplicities of all eigenvalues of $A_0$ are 1.

\hspace{55pt}(ii) If $x=(x_1, \ldots, x_m)^T$ is an eigenvector of $A_0$, then $x_1 \neq 0$.
\end{lemma}
\begin{proof}
Part (i) actually follows from part (ii). In fact, if there exists an eigenvalue with multiplicity bigger than 1, then
since $A_0$ is symmetric, there exist two linearly independent eigenvectors with same eigenvalue.
We can always make a linear combination of these two eigenvectors to obtain a new eigenvector whose first component is 0.
This contradicts part (ii).

Now we prove part (ii). Assume there exists an eigenvector $x=(x_1, \ldots, x_m)^T$ with $x_1 = 0$.
We want to show all $x_j=0$ by induction on $j$. Assume there exist $1 \leq i < m$ such that
$x_1=\cdots = x_i =0$. Since $A_0$ is positive definite by Proposition \ref{prop:  main result that A is symmetric}
and $x$ is an eigenvector, there exists a positive number $\lambda$ such that
$A_0 x = \lambda x$. By Proposition \ref{prop: explicit form of A_0 for Gr(2,2m)},
the first component on both sides of this equation gives
\[ \sum_{k=i+1}^m A_{1, k} \,\, x_k = \sum_{k=i+1}^m b_{k} x_k =0 \]
and the (i+1)-th component on both sides of equation $A_0 x = \lambda x$ gives
\[ x_{i+1} = \lambda^{-1} \sum_{k=i+1}^m A_{i+1, k} \,\, x_k = \lambda^{-1} (2i+1) \sum_{k=i+1}^m b_{k} x_k =0. \]
By induction, $x=0$, which contradicts to the assumption that $x$ is an eigenvector.
The lemma is thus proved.
\end{proof}

We can now give a precise description for $\mathfrak{F}_q = {\rm Span}_{\mathbb{C}[q,q^{-1}]} \{ \Delta^{*k} \mid k \geq 0\}$.

\begin{theorem}\label{thm: case of Gr(2,n)}
    For $X=\text{Gr}(2,n)$,
    \begin{equation} \label{eqn:FGr2n}
 \mathfrak{F}_q = \bigoplus_{\substack{j \equiv 0\;(mod\;D_2) \\ 0 \leq j \leq n-1}}   V_j \otimes \mathbb{C}[q,q^{-1}],
\end{equation}
where $D_2 = {\rm gcd}(n, 4)$ and
$V_j := {\rm Span}_{\mathbb{C}} \{ \sigma_\lambda
     \mid \lambda \in \mathcal{P}_{2n}, |\lambda| \equiv j ({\rm mod} \,\,\, n)\}$.
In particular,  $\mathfrak{F}_q= QH^*(X)_q$ if $n$ is odd.
\end{theorem}

\begin{proof}
    Let $D_1 = {\rm gcd}(n, 2)$. By equation \eqref{eqn: Omega_per is n-periodic},
    $[pt]^{* \frac{n}{D_1}} = q^{2(n-2)/D_1} \mathbf{1}$. Let
    \[ \Psi := (\Delta/[pt])^{*\frac{n}{D_1}} = q^{2(2-n)/D_1} \Delta^{*\frac{n}{D_1}}.\]
    Then the matrix of quantum multiplication by $\Psi$ on $V_0 \otimes  \mathbb{C}[q, q^{-1}]$ with respect to
    the basis $\Theta^q$ is $B:=A_0^{\frac{n}{D_1}}$.

    Since $A_0$ is a positive definite real symmetric matrix, so is $B$.
    Hence there exists orthonormal basis $\{e_1, \ldots, e_m\}$ of $\mathbb{R}^m$  and positive numbers
    $\lambda_1, \ldots, \lambda_m$ such that $B e_i = \lambda_i e_i$ for all $1 \leq i \leq m$.
    We can choose $\{e_1, \ldots, e_m\}$ to be also eigenvectors of $A_0$. By
    Lemma \ref{lem:A0Mult=1}, the first component of $e_i$ is not 0 for all $i$ and
    $\lambda_1, \ldots, \lambda_m$ are pairwise distinct positive numbers since
    eigenvalues of $A_0$ are pairwise distinct and positive.

    The coordinates of $\mathbf{1} \in V_0 \otimes \mathbb{C}[q, q^{-1}]$
    with respect to the basis $\Theta^q$ gives the vector $v:=(1,0,...,0)^T \in \mathbb{R}^m$.
    Write $v=\sum_{i=1}^m a_i e_i$. Then $a_i \neq 0$ since it is the first component of $e_i$ for all $i$.
    Hence $\{ a_i e_i \mid i=1, \ldots, m\}$ is also a basis of $\mathbb{R}^m$.
    The vector of coordinates of $\Psi^{*k} = \Psi^{*k} * \mathbf{1}$ with respect to the basis $\Theta^q$ is
    $B^k v = \sum_{i=1}^m a_i B^k e_i= \sum_{i=1}^m a_i \lambda_i^k e_i$.
    So the coordinates of $\Psi^{*k}$ with respect to the basis $\{ a_i e_i \mid i=1, \ldots, m\}$ is
    $(\lambda_1^k, \ldots, \lambda_m^k)$.
    Since $\lambda_1, \ldots, \lambda_m$ are pairwise distinct,
    the Vandermonde matrix

    \[
    \begin{bmatrix}
        1 & 1 & \cdots & 1\\
        \lambda_1 & \lambda_2 & \cdots & \lambda_m \\
        \vdots&\vdots&\ddots&\vdots\\
        \lambda_1^{(m-1)} & \lambda_2^{(m-1)} & \cdots & \lambda_m^{(m-1)}
    \end{bmatrix}
    \]

    \noindent is invertible.
    Thus $\{B^k v \mid 0\leq k \leq m-1\}$ is also a basis of $\mathbb{R}^m$.
    This implies that $\{ \Psi^{* k} \mid 0\leq k \leq m-1\}$ is a basis of $V_0\otimes\mathbb{C}[q,q^{-1}]$.
    Therefore  $V_0\otimes\mathbb{C}[q,q^{-1}] \subset \mathfrak{F}_q$.

    By Proposition \ref{prop:  main result that A is symmetric}, $\Delta$ is invertible. Since $\deg(\Delta)=\dim(X)$,
    by the proof of Lemma \ref{lem: same of cardinality}, for any $j\equiv0\;(mod\;D_2)$ and $0 \leq j \leq n-1$
    there exists non-negative integer $k$ such that quantum multiplication by $\Delta^{*k}$
    gives an isomorphism between $V_0\otimes\mathbb{C}[q,q^{-1}]$ and $V_j\otimes\mathbb{C}[q,q^{-1}]$.
    Hence $V_j \otimes\mathbb{C}[q,q^{-1}] \subset \mathfrak{F}_q$.
    Therefore equation \eqref{eqn:FGr2n} follows from Theorem \ref{thm: subspace spanned by Delta^i}.

    When $n$ is odd, $D_2=1$. So equation \eqref{eqn:FGr2n} implies that
     $\mathfrak{F}_q = QH^*(X)_q$. We thus finish the proof of this theorem
     and also complete the proof for Theorem~\ref{thm:  main thm 2}.
\end{proof}

\begin{corollary}
    For $\text{Gr}(2,n)$, $\dim (\mathfrak{F}) = \frac{n}{\text{gcd}(4,n)}\cdot\lfloor\frac{n}{2}\rfloor$.
\end{corollary}

\begin{proof} Since $\lfloor\frac{2(n-2)}{n}\rfloor=1$ for $n\geq 4$,
   by Theorems \ref{thm: estimation of dimension of Span(Delta^i)} and Theorem \ref{thm: case of Gr(2,n)},
   \[ \dim (\mathfrak{F}) =\frac{n}{\text{gcd}(4,n)}(p(0 \mid n-2,2)+ p(n \mid n-2,2)).\]
    Moreover $p(0 \mid n-2,2)=1$ since it is the number of partitions $\lambda$ with $|\lambda|=0$.
    $p(n \mid n-2,2)$ is the number of partitions $\lambda$ with $|\lambda|=n$ and $\lambda \in \mathcal{P}_{2n}$. All such partitions are given by $(n-i, i)$ for $2 \leq i \leq \lfloor \frac{n}{2} \rfloor$.
    So $p(n \mid n-2,2)=\lfloor\frac{n}{2}\rfloor-1$. The corollary is thus proved.
\end{proof}

In contrast, for $X= \text{Gr}(2,n)$, $\dim H^*(X) = \frac{1}{2} n (n-1)$.
It follows that $\lim_{n \rightarrow \infty} D_2 \frac{\dim (\mathfrak{F})}{\dim H^*(X)} = 1$.
So for large $n$, $\frac{\dim (\mathfrak{F})}{\dim H^*(X)}$ is approximately equal to $1/D_2 = 1/{\text{gcd}(4,n)}$.

\section{Fano complete intersection}\label{sec: fano complete intersection}

The handle element $\Delta$ for Fano complete intersections was computed by Cela in \cite{Ce}, which solves a conjecture
by Buch and Pandharipande in \cite{BuP}. In these papers, $\Delta$ was called the quantum Euler class.
In this section, we use Cela's formula for $\Delta$ to study complexity of quantum cohomology of Fano complete intersections. In particular, we will prove the following result for arbitrary reference state:
\begin{theorem} \label{thm:ComplexityFCI}
Let $X \subset \mathbb{P}^{r+L}$ be a Fano complete intersection of dimension bigger than $2$ whose total degree (i.e. $|\textbf{m}|$ defined below) is strictly less than $r+L$, then $\mathfrak{S}_\infty$ is empty.
If the total degree of $X$ is equal to $r+L$, then  $\mathfrak{S}_\infty$ contains at most 2 points.
\end{theorem}

\noindent
This theorem follows from a combination of Propositions \ref{prop: m<r+L finite states}, \ref{prop: m=r+L hprim=0 finite states}, \ref{prop: Sinfty for last case}, and remarks following these results.

\subsection{Quantum Cohomology of complete intersections}

We first recall some basic facts about quantum cohomology of Fano complete intersections.
We will follow the notations in \cite{BuP} and \cite{Ce}.

Let $X \subset \mathbb{P}^{r+L}$ be an $r$-dimensional smooth {\it complete intersection}
which can be defined as the
common zero locus
of $L$ homogeneous polynomials with degrees $\textbf{m}:=(m_1,m_2,\cdots,m_L)$.
If one of the degrees $m_i=1$, $X$ can be embedded into a lower dimensional projective space. Hence we will assume
that $m_i\geq 2$ for all $i$ and $X$ is not contained in a proper projective subspace of $\mathbb{P}^{r+L}$ . We will also assume $\dim(X)=r \geq 3$, which implies $H_2(X, \mathbb{Z}) \cong \mathbb{Z}$.
We will use the following notations

\[|\textbf{m}| :=\sum_{i=1}^L m_i,  \hspace{10pt}
\textbf{m}^{a\textbf{m}+b} :=\prod_{i=1}^L m_i^{am_i+b},  \hspace{10pt}
\textbf{m}! :=\prod_{i=1}^Lm_i!,\]

\noindent where $a,b\in \mathbb{Z}$.
We will assume $X$ is {\it Fano}, which is equivalent to assume $|\textbf{m}|\leq r+L$.

The vector space $H^*(X)$ has a decomposition
\[ H^*(X) = H^*(X)^{\text{res}} \oplus H^r(X)^{\text{prim}}, \]
where
$H^r(X)^{\text{prim}}$ is the set of all  cohomology classes of real degree $r$ which are annihilated
by the ordinary cup product with the hyperplane class $H \in H^2(X)$,
and $H^*(X)^{\text{res}} = {\rm Span} \{\mathbf{1}, H, \ldots, H^r\}$ is the image of the restriction map from $H^*(\mathbb{P}^{r+L})$ to $H^*(X)$.
$H^*(X)^{\text{res}}$ is called the {\it restricted part} (or ambient part) of $H^*(X)$,
and $H^r(X)^{\text{prim}}$ is called the {\it primitive part}.
From the above decomposition, we have
\begin{equation}\label{eq: dim of prim cohomology}
    \text{dim}H^r(X)^{\text{prim}}=(-1)^r(\chi(X)-(r+1)),
\end{equation}
where $\chi(X) $ is the Euler characteristic number of $X$.

The (complex) degree of the quantum parameter in $QH^*(X)= H^*(X) \otimes \mathbb{C}[q]$ is given
by (see, for example, Section 5.2 in \cite{BuP})
\begin{equation} \label{eqn:tauCIS}
 \tau := \deg(q)= r+L+1-|\textbf{m}|.
\end{equation}
Let $H^i$ and $H^{*i}$ be the $i$-th  powers of $H$ with respect to cup and quantum product respectively.
Define
    \[QH^*(X)^{\text{res}}:=\text{Span}_{\mathbb{C}}\{1,H,\cdots,H^r\}\otimes \mathbb{C}[q].\]
By a result of Graber, this space is closed under quantum multiplication (see Proposition 5.1 in \cite{BuP}).
By a result of Givental (\cite{G}), if $|\textbf{m}|\leq r+L-1$,
    \begin{equation}\label{eq: formula for H*(r+1) in |m| leq r+L-1}
        H^{*(r+1)}=\textbf{m}^{\textbf{m}}q H^{*(|\textbf{m}|-L)}.
    \end{equation}
If $|\textbf{m}|=r+L$,
    \begin{equation}\label{eq: formula for (H+m!q)*(r+1) in m=r+L}
        \widehat{H}^{*(r+1)} = \textbf{m}^{\textbf{m}} q \widehat{H}^{*r},
    \end{equation}
where $\widehat{H}:= H+\textbf{m}!q \mathbf{1}$.
By Corollary 5.3 in \cite{BuP}, if $|\textbf{m}| \leq r+L-1$, then
\begin{equation} \label{eqn:DeltaPrim=0}
 H * A = 0
\end{equation}
for all $A \in H^r(X)^{\text{prim}}$.
The following formulas for the handle element $\Delta$ (called the quantum Euler class in \cite{Ce})  were proved by Cela.
\begin{theorem}(Theorem 5 and Lemma 18 in \cite{Ce}) \\
\noindent
    If $|\textbf{m}|\leq r+L-1$, then
    \begin{equation}\label{eq: formula for Delta in |m| leq r+L-1}
        \Delta=\textbf{m}^{-1}\chi(X)H^{*r}+(\tau -\chi(X))\textbf{m}^{\textbf{m}-1}qH^{*(|\textbf{m}|-L-1)}.
    \end{equation}
    If $|\textbf{m}|=r+L$, then
    \begin{equation}
        \Delta=\textbf{m}^{-1}\chi(X) \widehat{H}^{*r}
           +\sum_{j=1}^{r} \left( \zeta(X)- \delta_{j1} \textbf{m}^{\textbf{m}-1}r \right) (\textbf{m}!)^{j-1} q^j \widehat{H}^{*(r-j)},
                                \label{eq: form of delta in H+m!q}
    \end{equation}
   where $\zeta(X):=\textbf{m}^{-1}(r+1-\chi(X))(\textbf{m}^{\textbf{m}}-\textbf{m}!)$.
\end{theorem}

Note that since $\textbf{m}=(m_1, \ldots, m_L)$ with all $m_i \geq 2$, we have $\textbf{m}^{\textbf{m}}-\textbf{m}!>0$. Hence $\zeta(X)=0$ if and only if $\chi(X)=r+1$. By Equation \eqref{eq: dim of prim cohomology}, this is equivalent to $H^r(X)^{\text{prim}}=0$.

\begin{remark}\label{rm: semisimp of fci}
     By equations \eqref{eqn:DeltaPrim=0} and \eqref{eq: formula for Delta in |m| leq r+L-1}, if
    $L+1 < |\textbf{m}| \leq r+L-1$, then quantum multiplication by $\Delta$ annihilates primitive classes. So $\Delta$ is not invertible. Hence in this case the small quantum cohomology of $X$ is not semisimple by Abrams' theorem in \cite{A}.
    By a result in \cite{Hu}, there are also many examples of Fano complete intersections whose restricted part of
    the big quantum cohomology are not semisimple.
\end{remark}

\subsection{Case of $|\textbf{m}|\leq r+L-1$}

By assumption, $m_i \geq 2$ for $1 \leq i \leq L$. Hence $|\textbf{m}| \geq L+1$. If
$|\textbf{m}| = L+1$, then L=1 and $\textbf{m}=(2)$. So $X$ is a quadric, which has been discussed in
Section \ref{subsect:  even quadrics}.
Therefore we can assume $|\textbf{m}| > L+1$.
In this case, $2 \leq \tau < r$.

By equation \eqref{eq: formula for Delta in |m| leq r+L-1},
$\Delta = a_1 H^{*r} + a_2 q H^{*\kappa}, $
where
\[
a_1 := \textbf{m}^{-1}\chi(X), \hspace{10pt}
a_2 := (\tau -\chi(X))\textbf{m}^{\textbf{m}-1}, \hspace{10pt}
\kappa := |\textbf{m}|-L-1 \geq 1.\]
By equation \eqref{eqn:DeltaPrim=0}, quantum multiplication by $\Delta$ annihilates all primitive classes.
So we only need to consider the action of $\Delta$ on restricted part.

By equation \eqref{eq: formula for H*(r+1) in |m| leq r+L-1},
\begin{equation} \label{eqn:Hrtau}
H^{*(r+i)} = \textbf{m}^{\textbf{m}} q H^{*(\kappa+i)}
\end{equation}
 for all $i \geq 1$.
This implies
\[
\Delta * H^{*i} = (a_1   + a_2 \textbf{m}^{-\textbf{m}}) H^{*(r+i)} = \tau \textbf{m}^{-1} H^{*r} * H^{*i}
\]
for all $i \geq 1$. Repeatedly using this formula, we obtain
\begin{equation} \label{eqn:DeltaPowerCIS1}
 \Delta^{*j} = \left( \tau \textbf{m}^{-1} \right)^j  H^{*j r}
\end{equation}
for all $j \geq 2$.

Now we set $q=1$ and consider the complexity with reference state $S_0=[\mathbf{1}] \in \mathbb{P}H^*(X)$.
The following result gives a precise description for states
with finite complexity.
\begin{proposition}\label{prop: m<r+L finite states}
    If $|\textbf{m}|\leq r+L-1$, the set of states in $\mathbb{P}H^*(X)$ with finite complexity is a finite set which is
    equal to
    \[F :=\{[\mathbf{1}],[\Delta]\}\cup \left\{ [H^{*jd}] \,\, |\,\,\, \kappa/d < j \leq  r/d \right\},\]
     where $d= {\rm gcd}(r,\tau)$. In particular, $\mathfrak{S}_\infty$ is an empty set.
\end{proposition}

\begin{proof}
By equation \eqref{eqn:DeltaPowerCIS1}, $[\Delta^{*j}] = [H^{*j r}] \in \mathbb{P}H^*(X)$ for all $j \geq 2$.
So the set of all states with finite complexity is
\[ F' :=\{ [\mathbf{1}], [\Delta], [H^{*j r}] \mid j \geq 2\}.\]
Since  $1 \leq \kappa= r-\tau \leq r-2$, by equation \eqref{eqn:Hrtau},
if $n = s \tau+j$ for any $s \geq  0$ and $\kappa+1 \leq j \leq r$, then
\begin{equation} \label{eqn:Hntaui}
[H^{*n}] = [H^{*j}] \in \mathbb{P}H^*(X).
\end{equation}
For any $l \geq 2$, we can find $s \geq  0$ and $\kappa < j \leq r$ such that $l r = s \tau +j$.
Since $d$ divides both $\tau$ and $r$, it also divides $j$. So $j = id$ for some integer $\kappa/d < i \leq r/d$.
By equation \eqref{eqn:Hntaui},
$[H^{*lr}] = [H^{*id}] \in F$. Hence $F' \subset F$.

On the other hand, since $d= {\rm gcd}(r, \tau)$, we can find integers $a > 0$ and $b < 0$ such that
$ar+b\tau=d$.
For any integer $\kappa/d < j \leq r/d$, $jar=(-jb)\tau + jd$. So by equation \eqref{eqn:Hntaui},
$[H^{*jd}] = [H^{*(ja)r}] \in F'$. This shows that $F \subset F'$. Hence we have $F'=F$.
The proposition is thus proved.
\end{proof}

\begin{remark} \label{rem:M<LR-1}
A direct consequence of the above Proposition is that if $\tau \neq \chi(X)$, then $\dim(\mathfrak{F})= 2+ \tau/d$ and
$\mathfrak{F} ={\rm Span} \{\mathbf{1}, H^{*jd} \mid \kappa/d \leq j \leq  r/d \}$ .
If $\tau=\chi(X)$, then $\dim(\mathfrak{F})= 1 + \tau/d$ and
$\mathfrak{F} ={\rm Span} \{\mathbf{1}, H^{*jd} \mid \kappa/d < j \leq  r/d \}$.
Moreover,  if $S_0=[z]$ is an arbitrary reference state, then the set of states with finite complexity is
$\{ [x*z]  \mid  [x] \in F, \,\, x * z \neq 0\}$, which is again a finite set. So $\mathfrak{S}_\infty$ is also empty.
\end{remark}

\subsection{Case of $|\textbf{m}|=r+L$}

In this case, $\tau=1$ and we choose $\{1,\widehat{H},\ldots,\widehat{H}^{*r}\}$ as a basis of $QH^*(X)^{\text{res}}$,
where $\widehat{H}=H+\textbf{m}! q \mathbf{1}$.

\subsubsection{Case with $\chi(X) = r+1$}

In this case $\zeta(X)=0$ and $H^r(X)^{\text{prim}}=0$. By equation~\eqref{eq: form of delta in H+m!q},
\[ \Delta=\textbf{m}^{-1} (r+1) \widehat{H}^{*r}
            -  \textbf{m}^{\textbf{m}-1}r  q \widehat{H}^{*(r-1)}. \]
By equation \eqref{eq: formula for (H+m!q)*(r+1) in m=r+L},  we have
$\Delta*\widehat{H}^{*j}
     = \textbf{m}^{j\textbf{m}-1}q^j \widehat{H}^{*r}
$
for all $j\geq1$. Repeatedly using this formula, we have
\begin{align*}
    \Delta^{*j}
    &=\textbf{m}^{(r\textbf{m}-1)(j-1)-1} q^{(j-1)r} \widehat{H}^{*r}
\end{align*}
for all $j \geq 2$. Hence we have

\begin{proposition}\label{prop: m=r+L hprim=0 finite states}
Set $q=1$ and consider the complexity with reference state $S_0 = [\mathbf{1}]$.
    If $|\textbf{m}|=r+L$ and $\chi(X)=r+1$, then the set of states with finite complexity is
$F = \{[\mathbf{1}], \,\, [\Delta], \,\, [\widehat{H}^{*r}]\}$. Moreover, $\mathfrak{S}_\infty$ is empty,
$\mathfrak{F}={\rm Span}\{\mathbf{1}, \,\, \widehat{H}^{*(r-1)}, \,\, \widehat{H}^{*r}\}$  has dimension equal to 3.
\end{proposition}

\begin{remark} \label{rem:SHres=0}
If $S_0=[z]$ is an arbitrary reference state, then the set of states with finite complexity is
$\{ [x*z]  \mid [x] \in F, \, \, x * z \neq 0\}$, which is again a finite set. So $\mathfrak{S}_\infty$ is also empty.

\end{remark}

\subsubsection{Case with $\chi(X) \neq r+1$}

In this case $\zeta(X)\neq 0$.
We rewrite equation \eqref{eq: form of delta in H+m!q} as
\begin{equation}
        \Delta=  \beta \mathbf{1}  + \xi \widehat{H}
           +\sum_{j=2}^{r-1} \left( \zeta(X)- \delta_{j,r-1} \textbf{m}^{\textbf{m}-1}r \right) (\textbf{m}!)^{r-j-1} q^{r-j} \widehat{H}^{*j} + \textbf{m}^{-1}\chi(X) \widehat{H}^{*r},
                                \label{eq: delta-beta}
\end{equation}
where $\beta:=\zeta(X)(\textbf{m}!)^{r-1}q^r$ and $\xi := \zeta(X)(\textbf{m}!)^{r-2}q^{r-1}$. We have $\beta \neq 0$ and
$\xi \neq 0$ if $q \neq 0$.

By Corollary 19 in \cite{Ce} and equation \eqref{eq: formula for (H+m!q)*(r+1) in m=r+L}, we have
\begin{align}
    \Delta * \widehat{H}^{*r}
    &=(\textbf{m}^{-1}-\textbf{m}^{-r\textbf{m}-1}(\textbf{m}!)^r(r+1-\chi(X))) \widehat{H}^{*2r}
    = \alpha  \widehat{H}^{*r},
          \label{eq: delta acts on H+m!q^r}
\end{align}
where
\[\alpha :=(\textbf{m}^{-1}-\textbf{m}^{-r\textbf{m}-1}(\textbf{m}!)^r(r+1-\chi(X)))(\textbf{m}^{\textbf{m}} q)^r.\]
We now compute $\Delta*\widehat{H}^{*(r-1)}$.
By equation \eqref{eq: formula for (H+m!q)*(r+1) in m=r+L}, we have
\begin{align}
    \Delta*\widehat{H}^{*(r-1)}
    =&\beta \widehat{H}^{*(r-1)} + \zeta(X)
             \sum_{j=1}^{r-1} (\textbf{m}!)^{r-j-1}(\textbf{m}^{\textbf{m}})^{j-1}
                   \,\, q^{r-1} \widehat{H}^{*r} \nonumber \\
    &+\textbf{m}^{(r-1)\textbf{m}-1}(\chi(X)-r)q^{r-1} \widehat{H}^{*r}. \nonumber
\end{align}
By definition of $\zeta(X)$, we have
\[ \zeta(X) \sum_{j=1}^{r-1} (\textbf{m}!)^{r-j-1}(\textbf{m}^{\textbf{m}})^{j-1}
     = \textbf{m}^{-1}(r+1-\chi(X))(\textbf{m}^{(r-1)\textbf{m}}-(\textbf{m}!)^{r-1}).\]
So
\begin{align}
\Delta*\widehat{H}^{*(r-1)}
    =&\beta \widehat{H}^{*(r-1)} + \omega \widehat{H}^{*r},
                         \label{eq: delta acts on H+m!q^r-1}
\end{align}
where $\omega:= \left( \textbf{m}^{(r-1)\textbf{m}}-(r+1-\chi(X))(\textbf{m}!)^{r-1} \right) \textbf{m}^{-1} q^{r-1}$.

By equations \eqref{eq: delta-beta} and \eqref{eq: formula for (H+m!q)*(r+1) in m=r+L}, for $0 \leq j \leq r-2$, we have
\begin{equation} \label{eqn:DeltaHj<r-1}
 \Delta*\widehat{H}^{*j} = \beta \widehat{H}^{*j}  + \xi \widehat{H}^{*(j+1)} + \sum_{k=j+2}^r c_k  \widehat{H}^{*j}
\end{equation}
for some $c_k \in \mathbb{Q}[q]$. The precise value of $c_k$ will not be used in proofs below.

Let $A=(A_{ij})_{(r+1)\times(r+1)}$ be the matrix of quantum multiplication by
$\Delta$ on $H^*(X)^{\text{res}}$ with respect to the basis
$\{ \widehat{H}^{*(r-j)} \mid 0 \leq j \leq r \}$, i.e.
\[
\Delta*\widehat{H}^{*(r-j)}=\sum_{i=0}^rA_{ij}\widehat{H}^{*(r-i)}.
\]

\noindent
By Equations  \eqref{eq: delta acts on H+m!q^r}, \eqref{eq: delta acts on H+m!q^r-1}, and \eqref{eqn:DeltaHj<r-1}, $A$ is of the following form

\begin{equation}
    A= \begin{bmatrix}
        \alpha & \omega & * & \cdots & * &*\\
        0 & \beta & \xi & \cdots & * &*\\
        0 & 0 & \beta & \cdots & * & *\\
        \vdots&\vdots&\vdots&\ddots&\vdots & \vdots\\
        0&0&0&\cdots& \beta & \xi\\
        0&0&0&\cdots& 0 & \beta
    \end{bmatrix}.       \label{mat: matrix form of multiply delta}
\end{equation}


To consider complexity, we set $q=1$.
Since $A$ is upper triangular, we can see immediately that the only eigenvalues of $A$ are $\alpha$ and $\beta$
with algebraic multiplicity $1$ and $r$ respectively. Since $\alpha$ and $\beta$ are real numbers, by Theorem \ref{thm:RealEigenvalue} in the appendix, $\mathfrak{S}_{\infty}$ contains at most 2 points
if the reference state $S_0 = [z]$ for some $z \in H^*(X)^{\text{res}}$.
To prove this is also true for arbitrary reference states, we need the following property for the Jordan canonical form of $A$.

\begin{lemma}\label{lem: jordan form of A is size r}
    The Jordan canonical form of matrix $A$ contains a Jordan block of size at least $r$.
\end{lemma}

\begin{proof}
    Let $I$ be the identity matrix. If $\alpha =\beta$, then $A-\beta I$ is an upper triangular matrix with all
diagonal entries equal to 0. By a straightforward induction on $i$, we can show that
all entries of $(A-\beta I)^i$ below the  upper $i$-th subdiagonal are zero and all entries along the upper
$i$-th subdiagonal are $\xi^i$ except possibly the first entry.
This implies that $(A-\beta I)^{r-1} \neq 0$ since its $(2, r+1)$ entry is $\xi^{r-1} \neq 0$. Hence
$A$ must have a Jordan block of size bigger than or equal to $r$.

Assume $\alpha \neq \beta$. Then $\alpha$ is an eigenvalue of $A$ with
multiplicity equal to 1. Hence $A$ has a Jordan block of size 1 with eigenvalue $\alpha$.
By equation \eqref{eq: delta acts on H+m!q^r}, $\widehat{H}^{*r}$ is an eigenvector of quantum multiplication
by $\Delta$ with eigenvalue $\alpha$. Hence the quantum multiplication
by $\Delta$ induces a linear automorphism, denoted by $\Psi$,
on vector space $H^*(X)^{\rm res} / \mathbb{C} \widehat{H}^{*r}$.
Let $\pi: H^*(X)^{\rm res}  \longrightarrow H^*(X)^{\rm res} / \mathbb{C} \widehat{H}^{*r}$
be the canonical projection map. Then the matrix of $\Psi$ with respect to the basis
$\{ \pi(\widehat{H}^{*i}) \mid 0 \leq i \leq r-1\}$ is of the form
    \[
    \bar{A} := \begin{bmatrix}
        \beta & \xi &*&\cdots & * &*\\
        0& \beta &\xi &\cdots & * &*\\
        0&0& \beta &\cdots & * &*\\
        \vdots&\vdots&\vdots&\ddots&\vdots & \vdots \\
        0&0&0&\cdots& \beta & \xi \\
        0&0&0&\cdots & 0 & \beta
    \end{bmatrix}.
    \]
Note that $\bar{A}- \beta I$ is an $r \times r$ upper triangular matrix with all diagonal entries equal to $0$.
Hence the $(1, r)$ entry of $(\bar{A}- \beta I)^{r-1}$ is $\xi^{r-1} \neq 0$. Consequently $(\bar{A}- \beta I)^{r-1} \neq 0$.
So $\Psi$ has a single Jordan block of size $r$.

On the other hand,
let $\{\widehat{H}^{*r},v_1,\cdots,v_r\}$ be a basis of $H^*(X)^{\rm res}$
with respect to which  the matrix of quantum multiplication by $\Delta$ is a Jordan canonical form.
Then $\{\pi(v_1), \cdots, \pi(v_r)\}$ is also a basis of $H^*(X)^{\rm res} / \mathbb{C} \widehat{H}^{*r}$.
The matrix of $\Psi$ with respect to this basis  is just the Jordan block of $\Delta$ corresponding to eigenvalue $\beta$,
hence it must have size $r$ as proved in the previous paragraph.
It follows that $\Delta$, or equivalently $A$, has a Jordan block of size $r$. The lemma is thus proved.
\end{proof}

Now we consider an arbitrary reference state $S_0=[z]$ for $z \in H^*(X)$.

\begin{lemma} \label{lem:JordanBz*}
Suppose $W$ is a subspace of $H^*(X)^{\text{res}}$ which is invariant under quantum multiplication by $\Delta$
and the Jordan canonical form of $\Delta$ restricted to $W$ consists of single Jordan block with eigenvalue $\lambda \neq 0$.
For any $z \in H^*(X)$, either $z * W = \{0\}$, or the Jordan canonical form of $\Delta$ restricted to $z * W$
also consists of single Jordan block with eigenvalue $\lambda$.
\end{lemma}

\begin{proof}
Since $W$ is invariant under quantum multiplication by $\Delta$,
so is $z*W$.
Assume $z * W \neq \{0\}$. Let $\{v_1, \ldots, v_n\}$ be a basis of $W$ with respect to which the matrix of quantum multiplication by $\Delta$ is a Jordan block. Then
\[ \Delta * v_1 = \lambda v_1, \hspace{20pt} \Delta * v_i = \lambda v_i + v_{i-1} \]
for $2 \leq i \leq n$.
Let $k$ be the smallest integer such that $z * v_k \neq 0$. Then
\[ \Delta * (v_k * z) = \lambda (v_k * z), \hspace{20pt} \Delta * (v_i*z) = \lambda (v_i*z) + v_{i-1}*z \]
for $k+1 \leq i \leq n$.
By Lemma \ref{lem:1JordanBlock}, vectors $v_i * z$ with $k \leq i \leq n$ are linearly independent, and therefore
form a basis of $z * W$. The above formulas also show that with respect to this basis, the
matrix of quantum multiplication by $\Delta$ is a Jordan block with eigenvalue $\lambda$.
The lemma is thus proved.
\end{proof}

\begin{lemma}\label{lem: eigenv of delta are real}
    For any $z \in H^*(X)$, if $z*H^*(X)^{\text{res}} \neq \{0\}$, then all eigenvalues of quantum multiplication  by $\Delta$ on $z*H^*(X)^{\text{res}}$ are real numbers.
\end{lemma}

\begin{proof}
    By Lemma \ref{lem: jordan form of A is size r}, the quantum multiplication  by $\Delta$ on $H^*(X)^{\text{res}}$
    has a Jordan block $J$ with size at least $r$. Since $\beta$ is an eigenvalue of algebraic multiplicity equal to $r$
    and $\dim H^*(X)^{\text{res}} = r+1$,
    the eigenvalue of $J$ must be $\beta$ which is not 0 by assumption.
    If the size of $J$ is $r+1$, then this action has a single Jordan block. So by Lemma \ref{lem:JordanBz*},
    the action of $\Delta$ on $z*H^*(X)^{\text{res}}$ also has a single Jordan block and its only eigenvalue is $\beta$.

    Suppose the size of $J$ is $r$. Then the action of $\Delta$ on $H^*(X)^{\text{res}}$ also has a Jordan block of size 1
    with eigenvalue $\alpha$.
    So $H^*(X)^{\text{res}}$ can be decomposed as a direct sum $\mathbb{C} v_0 \oplus W$, where $v_0$ is an eigenvector
    of $\Delta$ with eigenvalue $\alpha$, $W$ is invariant under the action of $\Delta$ and the restriction of this action
     to $W$ has a single
    Jordan block equal to $J$. By Lemma \ref{lem:JordanBz*}, the action of $\Delta$ on $z*W$ also has a single Jordan block with eigenvalue $\beta$. Note that $\Delta*(z * v_0) = \alpha(z * v_0)$.
    If $z * v_0 \notin z*W$, then the decomposition $z*H^*(X)^{\text{res}} = \mathbb{C} (z * v_0) \oplus (z*W)$ induces the Jordan decomposition of the action of $\Delta$ on $z*H^*(X)^{\text{res}}$. Hence eigenvalues of this action are either
    $\alpha$ or $\beta$. If $z * v_0 \in z*W$, then $z*H^*(X)^{\text{res}} = z*W$ and the only eigenvalue
    of the action by $\Delta$ on this space is $\beta$.
    In all above cases, eigenvalues are real numbers. This proves the lemma.
\end{proof}

Since $z \in z*H^*(X)^{\text{res}}$, an immediate consequence of Lemma \ref{lem: eigenv of delta are real} and Theorem \ref{thm:RealEigenvalue} is the following

\begin{proposition}\label{prop: Sinfty for last case}
   For any reference state $S_0=[z]$, $\mathfrak{S}_\infty$ contains at most two points.
\end{proposition}
This finishes the proof of Theorem \ref{thm:ComplexityFCI} and thus also completes the proof for
Theorem~\ref{thm:  main thm 1}.
We now consider the space $\mathfrak{F}$ defined by Equation \eqref{eqn:F}.

\begin{proposition} \label{prop:dimFCI}
We have
   $\dim(\mathfrak{F})=r+1$ and $\mathfrak{F}=H^*(X)^{\text{res}}$ if
\begin{equation} \label{eqn:omega=0}
    r+1-\chi(X) \neq \textbf{m}^{(r-1)\textbf{m}} (\textbf{m}!)^{1-r}.
\end{equation}
    Otherwise $\dim(\mathfrak{F})=r$.
\end{proposition}

\begin{proof}
Since $\mathbf{1} \in \mathfrak{F}$,
\begin{equation}\label{eq: a description of F}
   \mathfrak{F} = {\rm Span} \{\Delta^{*i} \mid  i \geq 0 \}
            = {\rm Span} \{(\Delta-\beta \mathbf{1})^{*i} \mid  i \geq 0 \}.
\end{equation}
For $1 \leq k \leq r$, let $W_k := {\rm Span}\{ \widehat{H}^{*i} \mid k \leq i \leq r\}$.
By equation \eqref{eq:  delta-beta},
\[ (\Delta-\beta \mathbf{1}) - \xi \widehat{H} \in W_2. \]
Hence by equation \eqref{eq: formula for (H+m!q)*(r+1) in m=r+L}, for all $1 \leq i \leq r-1$, we have
\[ (\Delta-\beta \mathbf{1})^{*i} - \xi^i \widehat{H}^{*i} \in W_{i+1}.\]
It follows that $(\Delta-\beta \mathbf{1})^{*i}$ for $0 \leq i \leq r-1$ are linearly independent
since $\xi \neq 0$.
So $\dim(\mathfrak{F}) \geq  r$.
Moreover,
\[ (\Delta-\beta \mathbf{1})^{*r} = C \widehat{H}^{*r} \in W_r \] for some constant $C$.
If $C = 0$, then $(\Delta-\beta \mathbf{1})^{*r} = 0$, which also implies $(\Delta-\beta \mathbf{1})^{*j} = 0$
for all $j  \geq r$. In this case, $\dim(\mathfrak{F}) =  r$.

If $C \neq 0$, then $(\Delta-\beta \mathbf{1})^{*i}$ for $0 \leq i \leq r$ are linearly independent. In this case
$\dim(\mathfrak{F}) \geq  r+1$.
Since $\mathfrak{F} \subset H^*(X)^{\text{res}}$, $\dim(\mathfrak{F}) \leq  \dim H^*(X)^{\text{res}} = r+1$.
Hence we must have $\dim(\mathfrak{F}) = r+1$ and $\mathfrak{F} = H^*(X)^{\text{res}}$.

To prove the proposition, we only need to show that  condition \eqref{eqn:omega=0} is equivalent to $C \neq 0$.
Note that condition \eqref{eqn:omega=0} is equivalent to $\omega \neq 0$, where $\omega$ is defined after equation
\eqref{eq: delta acts on H+m!q^r-1}. It is straightforward to check that this condition
is also equivalent to $\alpha \neq \beta$.

If $\alpha \neq \beta$, then by equation \eqref{eq: delta acts on H+m!q^r},
$(\Delta-\beta \mathbf{1}) * \widehat{H}^{*r} = (\alpha - \beta) \widehat{H}^{*r}$
and \[ (\Delta-\beta \mathbf{1})^{*r} *\widehat{H}^{*r} = (\alpha - \beta)^r \widehat{H}^{*r} \neq 0.\]
Hence $(\Delta-\beta \mathbf{1})^{*r} \neq 0$, which also implies $C \neq 0$.

If $\alpha = \beta$, then $\omega=0$. Since the matrix of quantum multiplication by $\Delta$ with respect to
the basis $\{\widehat{H}^{*(r-i)} \mid 0 \leq i \leq r\}$ is given by $A$ in
formula \eqref{mat: matrix form of multiply delta},
the matrix of quantum multiplication by $\Delta -\beta \mathbf{1}$ with respect to the same basis is $A- \beta I$,
which is upper triangular with all diagonal entries equal to 0.
The $(1, r+1)$ entry of $(A- \beta I)^r$ is $\omega \xi^{r-1}$ and all other entries are $0$. Since $\omega=0$,
$(A- \beta I)^r=0$. This implies $(\Delta-\beta \mathbf{1})^{*r} = 0$, and hence $C=0$.
The proposition is thus proved.
\end{proof}

\vspace{20pt}

\hspace{100pt} {\Large \bf Appendix}
\appendix

\section{A property for matrices with real eigenvalues}
\label{sec:RealEigenvalue}

Let $M$ be an $n \times n$ matrix which acts on $\mathbb{C}^n$ by matrix multiplication.
Fix any reference state $S_0 = [z] \in \mathbb{CP}^{n-1}$ where $z \in  \mathbb{C}^n$.
Let
\[ F := \{ [M^k z] \in \mathbb{CP}^{n-1} \mid k  \geq 0, \,\,\, M^k z \neq 0 \},  \hspace{10pt}
\mathfrak{S}_\infty := \overline{F} \setminus F \]
where  $\overline{F}$ is the closure of $F$ in $\mathbb{CP}^{n-1}$.
The main purpose of this appendix is to estimate the size of
$\mathfrak{S}_\infty$  if all eigenvalues of $M$ are real.

We first consider a special case where the Jordan decomposition of $M$ has a single Jordan block.

 \begin{lemma} \label{lem:1JordanBlock}
Suppose that there exist
$v_1, \ldots, v_n \in \mathbb{C}^n$ and $\lambda \in \mathbb{C}$ such that
 \[ M v_1 = \lambda v_1, \hspace{20pt} M v_i = \lambda v_i + v_{i-1}\]
  for $2 \leq i \leq n$ with $v_1\neq0$ and $\lambda \neq 0$. Then $v_1, \ldots, v_n$ are linearly independent. Furthermore,
 if $x = \sum_{i=1}^r x_i v_i$ for some $1 \leq r \leq n$, $x_i \in \mathbb{C}$, and $x_r \neq 0$,
 then
 \begin{equation} \label{eqn:1Bolck}
  \lim_{k \rightarrow \infty} \lambda^{-k} \binom{k}{r-1}^{-1} M^k x =  \lambda^{1-r} x_r v_1.
 \end{equation}
 \end{lemma}

 \begin{proof}
 We first prove equation \eqref{eqn:1Bolck}. A straightforward induction on $k$ shows
 \begin{equation} \label{eqn:Akx}
 M^k x = \sum_{i=1}^r \sum_{j=i}^r x_j \binom{k}{j-i} \lambda^{k-j+i} v_i.
 \end{equation}
 Hence
 \[ \lambda^{-k} \binom{k}{r-1}^{-1} M^k x
    = \sum_{i=1}^r \sum_{j=i}^r x_j \frac{\binom{k}{j-i}}{\binom{k}{r-1}} \lambda^{i-j} v_i.
 \]
 Since  $\lim_{k \rightarrow \infty} \binom{k}{a} / \binom{k}{b} = 0$ if $a < b$,  on the right hand side of the above equation,
 only the summand with $i=1$ and $j=r$ survives after taking limit as $k \rightarrow \infty$. This proves equation
 \eqref{eqn:1Bolck}.

 If $v_1, \ldots, v_n$ are linearly dependent, then there exist $x_1, \ldots, x_r \in \mathbb{C}$ for some  $1 \leq r \leq n$ with $x_r \neq 0$ such that $0 = \sum_{i=1}^r x_i v_i$. By equation \eqref{eqn:1Bolck}, we would have
 $\lim_{k \rightarrow \infty} \lambda^{-k} \binom{k}{r-1}^{-1} M^k 0 =  \lambda^{1-r} x_r v_1$.
 This is not possible since by assumption the right hand side of this equation is not $0$.
 Hence  $v_1, \ldots, v_n$ must be linearly independent. The lemma is thus proved.
 \end{proof}

\begin{theorem} \label{thm:RealEigenvalue}
If all eigenvalues of $M$ are real, then $\mathfrak{S}_\infty$ contains at most 2 points.
\end{theorem}

\begin{proof}
Recall the reference state is $S_0=[z] \in \mathbb{CP}^{n-1}$. If $M^k z = 0$ for some $k \geq 0$, then $F$ is a finite set and $\mathfrak{S}_\infty$ is empty.
Hence we can assume $M^k z \neq 0$ for all $k \geq 0$.

By Jordan decomposition theorem, $\mathbb{C}^n$ can be decomposed as
\[ \mathbb{C}^n = W_1 \oplus \cdots \oplus W_s \]
such that each $W_i$ is an $n_i$ dimensional subspace which is invariant under the action of $M$, and the restriction of the action of $M$ to $W_i$ has a single
Jordan block. In other words, there exists a number $\lambda_i$ and
a basis $\{ v_j^i \mid 1 \leq j \leq n_i \}$ of $W_i$
such that  $M v_1^i = \lambda_i v_1^i $ and $M v_j^i = \lambda_i v_j^i + v_{j-1}^i $ for all $2 \leq j \leq n_i$.
Since $\lambda_i$ is an eigenvalue of $M$, it must be a real number.

Write $z = x^1+ \cdots + x^s$ with each $x^i \in W_i$. If $x_i \neq 0$, then there exists unique $1 \leq r_i \leq n_i$ and
$x^i_j \in \mathbb{C}$ for $1 \leq j \leq r_i$ with $x^i_{r_i} \neq 0$ such that $x^i = \sum_{j=1}^{r_i} x^i_j v^i_j$.
Let
\begin{equation} \label{eqn:rlambda}
  \lambda := \max \{ |\lambda_i| \mid 1 \leq i \leq s, \,\,\, x_i \neq 0\}, \hspace{20pt}
 r := \max \{r_j \mid 1 \leq j \leq s, |\lambda_j| = \lambda\}.
\end{equation}
Then $\lambda \neq 0$ since $M^k z \neq 0$ for all $k \geq 0$. Hence
\[
\lambda^{-k} \binom{k}{r-1}^{-1} M^k z
= \sum_{i=1}^s  \left( \frac{\lambda_i}{\lambda} \right)^k  \frac{\binom{k}{r_i-1}}{\binom{k}{r-1}}
              \cdot \lambda_i^{-k} \binom{k}{r_i-1}^{-1} M^k x^i.
\]
Note that if $|\lambda_i| < \lambda$ then
\[ \lim_{k \rightarrow \infty} \left( \frac{\lambda_i}{\lambda} \right)^k  \frac{\binom{k}{r_i-1}}{\binom{k}{r-1}} =0.\]
If  $r_i < r$, then
$\lim_{k \rightarrow \infty}  \frac{\binom{k}{r_i-1}}{\binom{k}{r-1}} =0$.
Hence by Lemma \ref{lem:1JordanBlock}, any infinite convergent subsequence of
$\{\lambda^{-k} \binom{k}{r-1}^{-1} M^k z \mid k \geq 0 \}$
must converge to a point of the form
\[ \bigg(  \sum_{\lambda_i = \lambda, \,\, r_i = r}  \lambda_i^{1-r} x^i_r v^i_1 \bigg)
    \pm \bigg( \sum_{\lambda_i = - \lambda, \,\, r_i = r}    \lambda_i^{1-r} x^i_r v^i_1 \bigg). \]
Hence in $\mathbb{CP}^{n-1}$, we have
\begin{equation} \label{eqn:SRange}
 \mathfrak{S}_\infty \subset \bigg\{ \bigg[
     \bigg(  \sum_{\lambda_i = \lambda, \,\, r_i = r}  \lambda_i^{1-r} x^i_r v^i_1 \bigg)
    \pm \bigg( \sum_{\lambda_i = - \lambda, \,\, r_i = r}    \lambda_i^{1-r} x^i_r v^i_1 \bigg)
    \bigg] \bigg\}.
\end{equation}
The theorem is thus proved.
\end{proof}

\begin{remark} \label{rem:sizeSinftyM}
By equation \eqref{eqn:SRange}, if all eigenvalues of $M$ are non-negative, or if all eigenvalues of $M$ are non-positive,
then $\mathfrak{S}_\infty$ contains at most one point.
More generally, we can also consider $M$ with possibly complex eigenvalues. In this case, we can still
define $\lambda$ using equation \eqref{eqn:rlambda}. If for every  pair of eigenvalues $\lambda_i$ and $\lambda_j$ with
$|\lambda_i| = |\lambda_j| = \lambda$, there exists a positive integer $n_{ij}$ such that $(\lambda_i/\lambda_j)^{n_{ij}}=1$, then $\mathfrak{S}_\infty$ is
finite. This can be proved by a slight modification of the proof for Theorem \ref{thm:RealEigenvalue}.  Moreover,
we can also estimate the size of $\mathfrak{S}_\infty$ from these integers $n_{ij}$.
\end{remark}

\vspace{30pt} \noindent
Xiaobo Liu \\
School of Mathematical Sciences \& \\
Beijing International Center for Mathematical Research, \\
Peking University, Beijing, China. \\
Email: {\it xbliu@math.pku.edu.cn}
\ \\ \ \\
Chongyu Wang \\
School of Mathematical Sciences, \\
Peking University, Beijing, China. \\
Email: {\it wangcyu@pku.edu.cn}

\end{document}